\documentclass[11pt]{article}

\usepackage{amsmath}
\usepackage{amsfonts}
\usepackage{amssymb}
\usepackage{amsthm}
\usepackage{enumitem}
\usepackage{mathtools}
\usepackage{setspace}
\usepackage{etoolbox}

\newtheoremstyle{mystyle}
{11pt}				
{11pt}				
{}					
{}					
{\bfseries}			
{}					
{5.5pt}				
{}					

\theoremstyle{mystyle}
\newtheorem{theorem}{Theorem}[section]
\newtheorem{definition}[theorem]{Definition}

\newtheorem{proposition}[theorem]{Proposition}
\newtheorem{corollary}[theorem]{Corollary}
\newtheorem{example}[theorem]{Example}

\renewenvironment{proof}[1][Proof.]{\vspace{-16.5pt} \begin{trivlist}
	\item[\hskip \labelsep {\bfseries #1}]}{\qed \end{trivlist}}

\setitemize{noitemsep, topsep=5.5pt, parsep=5.5pt, partopsep=0pt}

\allowdisplaybreaks
\appto\normalsize{
	\abovedisplayskip=5.5pt plus 2pt minus 2pt
	\belowdisplayskip=5.5pt plus 2pt minus 2pt
	\abovedisplayshortskip=5.5pt plus 2pt minus 2pt
	\belowdisplayshortskip=5.5pt plus 2pt minus 2pt}
\appto\small{
	\abovedisplayskip=5.5pt plus 2pt minus 2pt
	\belowdisplayskip=5.5pt plus 2pt minus 2pt
	\abovedisplayshortskip=5.5pt plus 2pt minus 2pt
	\belowdisplayshortskip=5.5pt plus 2pt minus 2pt}

\newcommand{\gap}{\vspace{11pt}}

\newcommand{\conv}{\operatorname{conv}}

\newcommand{\R}{\mathcal{R}}
\newcommand{\Rn}{\mathcal{R}^n}
\newcommand{\Rnp}{\mathcal{R}_+^n}
\newcommand{\Sn}{\mathcal{S}^n}

\newcommand{\Hn}{\mathcal{H}^n}
\newcommand{\V}{{\cal V}}
\newcommand{\Z}{{\cal Z}}
\newcommand{\W}{{\cal W}}
\newcommand{\G}{{\cal G}}

\newcommand{\barx}{\overline{x}}
\newcommand{\wg}{\widetilde{\gamma}}
\newcommand{\wl}{\widetilde{\lambda}}

\newcommand{\lchi}{\raisebox{2pt}{$\chi$}}

\title{\bf Optimizing certain combinations of  spectral  and  \\linear/distance
 functions over    spectral sets }
\author{
 M. Seetharama Gowda\\
        Department of Mathematics and Statistics\\
        University of Maryland, Baltimore County\\
        Baltimore, Maryland 21250, USA\\
        gowda@umbc.edu}

\date{\today}

\begin{document}

\maketitle

\begin{abstract}
In the settings  of   Euclidean Jordan algebras, normal decomposition systems (or Eaton triples), and  structures 
induced by complete isometric hyperbolic polynomials, 
we consider the problem of  optimizing   a certain combination (such as  the sum) of  spectral and  
linear/distance  functions over a  spectral set. To present a unified theory, 
we introduce a new system called Fan-Theobald-von Neumann system which is a triple $(\V,\W,\lambda)$, 
where $\V$ and $\W$ are  real inner product spaces and $\lambda:\V\rightarrow \W$ is a norm preserving map 
satisfying a Fan-Theobald-von Neumann type inequality together with a condition for equality. 
In this general setting, we show that optimizing  a certain combination of   
spectral  and   linear/distance
 functions over a set of the form $E=\lambda^{-1}(Q)$ in $\V$, where $Q$ is a subset of $\W$,  is equivalent to optimizing a corresponding combination over the set $\lambda(E)$ and 
relate the attainment of the optimal value to a commutativity concept. We also study related results for convex functions in place of linear/distance 
functions. 
Particular instances include the classical results of Fan and  Theobald,  von Neumann,
results of Tam, Lewis, and Bauschke et al., and recent results of Ram\'{i}rez et al.
As an application, we present 
a commutation principle for variational inequality problems over such a system.
\end{abstract}

\vspace{1cm}
\noindent{\bf Key Words}: Fan-Theobald-von Neumann system, 
Euclidean Jordan algebra, normal decomposition system, Eaton triple, hyperbolic polynomial, 
spectral set, eigenvalue map, strong operator commutativity, variational inequality problem 
\\

\noindent{\bf AMS Subject Classification:}  
15A27, 17C20, 46N10, 52A41, 90C25, 90C33.
\newpage


\section{Introduction}
Let $\V$ and $\W$ be two real inner product spaces, $\lambda:\V\rightarrow \W$ be  a map, 
$Q$ be a subset of $\W$, and $E:=\lambda^{-1}(Q)$. In analogy with certain concepts in Euclidean Jordan algebras, 
we say that   $\lambda$  is an eigenvalue map and $E$ is a  spectral set; a function $\Phi:\V\rightarrow \R$ is said to 
be a  spectral function if it is of the form $\Phi=\phi\,\circ \lambda$ for some function $\phi:\W\rightarrow \R$. 
Given a real valued function on  $\V$, we are interested in reformulating the problem of optimizing (that is, finding the  
infimum/supremum/minimum/maximum of) this function  over  $E$
equivalently as a problem of optimizing a related function  over $\lambda(E)$ 
with the expectation that the latter problem is relatively easy to solve and/or gives some information on the former problem. For example, if $E$ is a spectral set and $\Phi$ is a spectral function on $\V$, then a mere change of 
variable will show that such a reformulation holds.
The main objective of this paper is to show that in some settings, such a reformulation can be carried out for
certain combinations of  $\Phi$ and linear/distance/convex functions. To elaborate, let $c\in \V$ and consider the 
linear function $f(x):=\langle c,x\rangle$ and the distance function $g(x):=||c-x||$ over $\V$. Define the corresponding 
linear and distance functions over $\W$:  $f^*(w)=\langle \lambda(c),w\rangle$ and
$g_*(w)=||\lambda(c)-w||.$
Then,  under certain conditions on the triple $(\V,\W,\lambda)$ -- thus defining a new system  called 
Fan-Theobald-von Neumann system -- 
we will show that 
\begin{equation}\label{basic identities}
\sup_{E}\,(f+\Phi)=\sup_{\lambda(E)}\,(f^*+\phi)\quad \mbox{and}\quad \inf_{E} (g+\Phi)=\inf_{\lambda(E)}(g_*+\phi),
\end{equation}
with similar statements where `supremum' is replaced by `maximum' and `infimum' by `minimum'.
Additionally, 
we relate the attainment in each of these problems to a  condition of the form
$$\langle c,x\rangle=\langle \lambda(c),\lambda(x)\rangle, $$
thus defining the concept of {\it commutativity} in this system. We also handle problems of the 
form $\inf_{E} (f+\Phi)$ and  $\sup_{E}(g+\Phi)$ in a similar way and  consider replacing the sum by other combinations, and replacing linear and distance functions by convex functions.   By specializing, we show that all these hold in 
Euclidean Jordan algebras \cite{faraut-koranyi}, normal decomposition systems \cite{lewis} (in particular, Eaton 
triples \cite{eaton-perlman}), and certain structures induced by complete hyperbolic polynomials \cite{bauschke et al}.
\\ 
As a simple illustration,   consider the following  optimization problem stated in the setting of  $\Sn$, the 
(Euclidean Jordan) algebra of  all $n\times n$ real symmetric matrices:
for a given $C\in \Sn$ ($n\geq 2$), find 
$$\max\big \{ \langle C,X\rangle: X\succeq 0,\, 1\leq \lambda_{max}(X)\leq 2\Big \},$$
where $X\succeq 0$ means that $X$  is positive semidefinite and $\lambda_{max}(X)$ is the maximum eigenvalue of $X$. 
Here, the objective function is linear and the constraint set (defined by eigenvalues) is nonconvex.
This problem turns out to be equivalent to finding $$\max\big \{ \langle
\lambda(C),q\rangle: q=(q_1,q_2,\ldots,q_n)\in \Rnp: q_1\geq q_2\geq \cdots\geq q_n,\, 1\leq q_1\leq 2\Big \},$$
where $\lambda(C)$ is the vector of eigenvalues of $C$ written in the decreasing order.
Clearly, the latter problem, stated in $\Rn$, is easier to solve than the former.
In this reformulation, $\V=\Sn$, $\W=\Rn$, $\lambda:\V\rightarrow \Rn$ is the eigenvalue map that takes any real symmetric matrix  
to its vector of eigenvalues written in the decreasing order, $E=\{X\in \Sn: X\succeq 0,\,1\leq \lambda_{max}(X)\leq 2\}$, 
and $\Phi=0$. Another illustrative example in the setting of $\Sn$ 
is the problem $\inf_{E}(f+\Phi)$, where $f(X)=\langle C,X\rangle$, $\Phi(X)=-\log\,\det(X)$, and 
$E$ is an appropriate spectral set.\\

A Fan-Theobald-von Neumann system (FTvN system, for short) introduced in this paper  
is a triple $(\V,\W,\lambda)$, where $\V$ and $\W$ are  real inner product spaces and 
$\lambda:\V\rightarrow \W$ is a norm preserving map satisfying the
property 
\begin{equation} \label{intro ftvn}
\max\Big \{\langle c,x\rangle:\,x\in [u]\Big \}=\langle \lambda(c),\lambda(u)\rangle\quad (c,u\in \V),
\end{equation}
with $[u]:=\lambda^{-1}(\{\lambda(u)\})$ denoting the (so-called)  $\lambda$-orbit of $u$,  
 see Section 2 for  an elaborated version   and a formulation in terms of the distance function.
The above property can be regarded as a combination of   Fan-Theobald-von Neumann type inequality, namely,
$\langle c,x\rangle\leq \langle \lambda(c),\lambda(x)\rangle$, and a ``commutativity" condition for equality.
Perhaps, the simplest nontrivial example of such a system is the triple $(\V,\R,\lambda)$, where $\V$ is a real inner product space and 
$\lambda$ denotes the corresponding norm; in this case, the defining property reduces to the Cauchy-Schwarz 
inequality together  with a condition for equality. When $\V$ is a Euclidean Jordan algebra of rank $n$ carrying the 
trace inner product with $\lambda:\V\rightarrow \Rn$ denoting the eigenvalue map, the triple $(\V,\Rn,\lambda)$ 
becomes a FTvN system. More generally, if $\V$ is a finite dimensional real vector space and $p$ is a real 
homogeneous polynomial of degree $n$ on $\V$ that is  hyperbolic with respect to an element $e\in \V$, complete, and isometric 
\cite{bauschke et al}, then $(\V,\Rn,\lambda)$ becomes a FTvN system, where for any element $x$ in $\V$, $\lambda(x)$ denotes the vector of 
roots of the univariate polynomial $t\rightarrow p(te-x)$ written in the decreasing order. Also, when 
$(\V,\G,\gamma)$ is a normal decomposition system (in particular, an Eaton triple) with $\G$ denoting a closed subgroup of the orthogonal group of a real inner product space $\V$ and $\gamma:\V\rightarrow \V$ satisfying some specified properties, the triple
$(\V,\W,\gamma)$ becomes a FTvN system, where $\W:=\mbox{span}(\gamma(\V))$. 
\\

The motivation for our work comes from  several results mentioned below.
\begin{itemize}
\item [$\bullet$]  
For two $n\times n$ complex Hermitian matrices $C$ and $A$ with eigenvalues 
$c_1\geq c_2\geq \cdots \geq c_n$ and $a_1\geq a_2\geq \cdots\geq a_n$, a classical result of Fan \cite{fan} (see also \cite{theobald,tam}) states that
$$\max \Big \{\mbox{tr}(CUAU^*):\,U\in {\cal C}^{n\times n} \,\,\mbox{is unitary} \Big\}= \sum_{i=1}^{n}c_ia_i,$$
where `tr' refers to the trace.
By working in the Euclidean Jordan algebra $\Hn$ (of all $n\times n$ complex Hermitian matrices) with $\langle X,Y\rangle=tr(XY)$ and $\lambda(X)$ denoting the eigenvalues of $X$ written in the decreasing order, this result can be viewed as 
describing the maximum of the linear function $f(X):=\langle C,X\rangle$ over the eigenvalue orbit 
$E=\lambda^{-1}(\{\lambda(A)\})=\{X\in \Hn: \lambda(X)=\lambda(A)\}$ with the optimal value given by $\langle \lambda(C),\lambda(A)\rangle$. Furthermore, the attainment is described in the form of simultaneous order eigenvalue decomposition.  
\item [$\bullet$]
A result of  von Neumann  
deals with two $n\times n$ complex matrices $C$ and $A$ with singular values $c_1\geq c_2\geq \cdots \geq c_n$ and $a_1\geq a_2\geq \cdots\geq a_n$. It asserts that  
$$\max\Big \{\mbox{Re\,tr}(CUA^*V): \,U,V\in {\cal C}^{n\times n} \,\,\mbox{are unitary}\Big \}= \sum_{i=1}^{n}c_ia_i.$$ 
By working in the normal decomposition system (Eaton triple) corresponding to the space $M_n$ (of all $n\times n$ complex matrices) with $\lambda(X)$  denoting the singular values of $X$ written in the decreasing order, and $\langle X,Y\rangle=\mbox{Re\,tr}(XY^*)$, we can view this result as 
describing the maximum of the linear function $f(X):=\langle C,X\rangle$ over the singular value orbit 
$E=\lambda^{-1}(\{\lambda(A)\})=\{X\in M_n:\lambda(X)=\lambda(A)\}$ with the optimal value given by 
$\langle \lambda(C),\lambda(A)\rangle$. Here, the attainment is described in the form of simultaneous order singular
 decomposition.

\item [$\bullet$]
In \cite{chu-driessel}, Chu and Driessel considered the problems of minimizing the distance function $||C-X||$ over the eigenvalue orbit of a 
matrix $A$ in $\Hn$ and over the  singular value orbit of a matrix $A$ in $M_n$.
	Their results were refined by Tam (\cite{tam}, Corollaries 2.2 and 2.3), Tam and Hill (\cite{tam-hill}, Theorem 27) by working in the setting of Lie algebras/Eaton triples.  
Related works 
\cite{li-tsing,holmes-tam} deal with  distance to the convex hull of eigenvalue/singular value orbits. 

\item [$\bullet$] In the setting of a  normal decomposition system $(\V,\G,\gamma)$, Lewis (\cite{lewis}, Proposition 2.3 and Theorem 2.4]) describes the property
        $$ \max_{A \in \G}\, \langle Ax,y\rangle = \langle \gamma(x),\gamma(y)\rangle $$
with a condition for equality: $\langle x,y\rangle = \langle \gamma(x),\gamma(y)\rangle $
if and only if there exists an $A \in \G$ such that $x = A\gamma(x)$ and $y = A\gamma(y)$.
This property can be viewed as a statement on maximizing a linear function over an orbit of the form $\{Ax:A\in \G\}$.

\item[$\bullet$]
In the setting of a triple $(\V,\Rn,\,\lambda)$ induced by a complete hyperbolic polynomial $p$ on a finite dimensional real vector space $\V$, Bauschke et al., \cite{bauschke et al} introduce the concept of `isometric hyperbolic polynomial' and state a result (\cite{bauschke et al}, Proposition 5.3) describing the maximum of a linear function over a $\lambda$-orbit.
In this result, the optimality condition is given in the form $\lambda(x+y)=\lambda(x)+\lambda(y)$.

\item [$\bullet$] In the setting of a simple Euclidean Jordan algebra $\V$ of rank $n$ carrying the
trace inner product with $\lambda:\V\rightarrow \Rn$ denoting the eigenvalue map, it is known that for any $c\in \V$, 
\begin{equation}\label{moldovan result}
\max\Big \{\langle c,x\rangle: x\in {\mathcal{J}}^k(\V)\Big \}=\lambda_1(c)+\lambda_2(c)+\cdots+\lambda_k(c),
\end{equation}
where ${\mathcal{J}}^k(\V)$ is the set of all idempotents of rank $k$ in $\V$, $1\leq k\leq n$, see \cite{moldovan}, Theorem 17. This result can be viewed as a statement on maximizing a linear function
over the $\lambda$-orbit of (any) one idempotent of rank $k$.

\item [$\bullet$]
In \cite{ramirez et al},  Ram\'{i}rez, Seeger, and Sossa formulate a commutation principle in the setting of Euclidean Jordan algebras:
If $a$ is a local optimizer of the problem 
$$\mbox{min/max}\,\{h(x)+\Phi(x):x\in E\},$$
 where $E$ is a spectral 
set, $\Phi$ is a spectral function, and $h$ is Fr\'{e}chet differentiable, then $a$ and $h^\prime(a)$ operator commute. 
Based on this, they present a commutation principle for a variational inequality problem and consider the problem of describing the distance to a spectral set. See \cite{gowda-jeong} for a slight weakening of the conditions and a similar result proved in the setting of normal decomposition systems.
See also \cite{niezgoda-commutation}  for certain elaborations and applications. 
\end{itemize}

Our contributions in this paper are as follows. Motivated by the above results, 
we formulate the definition of  
a Fan-Theobald-von Neumann system and study some of its basic properties. We show that Euclidean Jordan algebras, normal decomposition systems (in particular, Eaton triples), and structures induced by complete isometric hyperbolic polynomials are particular instances.
In this general framework, we describe results of the form (\ref{basic identities}) 
which extend/recover many of the above mentioned results.
We also introduce the concept of commutativity in a FTvN system that encompasses (or extends)  the concepts of simultaneous order eigenvalue/singular value decompositions and operator commutativity.  
Additionally, we present a commutation principle for a variational inequality problem in a FTvN system that 
even strengthens the commutation principle of Ram\'{i}rez, Seeger, and Sossa  (\cite{ramirez et al}, Proposition 8) stated in the setting of  
 Euclidean Jordan algebras. \\

The outline of the paper is as follows. In Section 2, we describe FTvN systems and present some basic properties.
Section 3 deals with  optimization problems coming from 
certain combinations of  spectral and  
linear/distance/convex functions, and a 
commutation principle for  variational inequality problems. 
Section 4 deals with Euclidean Jordan algebras and some specialized results.
Section 5 deals with an FTvN system induced by certain hyperbolic polynomials. 
In Section 6, we cover normal decomposition systems and Eaton triples.

\section{Fan-Theobald-von Neumann system}

Motivated by the results of Fan, Theobald, and von Neumann mentioned in the Introduction, we now formulate the definition of 
a Fan-Theobald-von Neumann system.
Let $\V$ and $\W$ be two real inner product spaces where, for convenience, we use the same inner product (and norm) 
notation. Let $\lambda:\V\rightarrow \W$ be a map. We define the  {\it $\lambda$-orbit of an element} $u\in \V$ as the set 
$$[u]:=\{x\in \V: \lambda(x)=\lambda(u)\}.$$ 
We have   the following elementary result.

\begin{proposition} \label{equivalence of linear and distance functions}
{\it Suppose $||\lambda(x)||=||x||$ for all $x\in \V$. Then, for any $u, c\in \V$, the following are equivalent:
\begin{itemize}
\item [$(a)$] $\max\Big \{\langle c,x\rangle:\,x\in [u]\Big \}=\langle \lambda(c),\lambda(u)\rangle.$
\item [$(b)$] $\min\Big \{||c-x||:x\in [u]\Big \}=||\lambda(c)-\lambda(u)||.$
\end{itemize}
}
\end{proposition}

\begin{proof}
 Fix $u,c\in \V$ and let $x\in [u]$. Then, $||x||=||\lambda(x)||=||\lambda(u)||$. Hence,
$$||c-x||^2=||c||^2+||x||^2-2\langle c,x\rangle=||\lambda(c)||^2+||\lambda(u)||^2-2\langle c,x\rangle.$$
As $x$ varies over $[u]$, we have
$$\min\, ||c-x||^2=||\lambda(c)||^2+||\lambda(u)||^2-2\max\,\langle c,x\rangle.$$ Comparing this to
$$||\lambda(c)-\lambda(u)||^2=||\lambda(c)||^2+||\lambda(u)||^2-2\langle \lambda(c),\lambda(u)\rangle$$
we see that $(a)$ holds if and only if $(b)$ holds.
\end{proof}

Note that condition $(a)$ deals with  the inner product induced linear function $x\rightarrow \langle c,x\rangle$ and 
attainment of its maximum over the $\lambda$-orbit $[u]$.
We now define a  Fan-Theobald-von Neumann system  as a triple $(\V,\W,\lambda)$ where $||\lambda(x)||=||x||$ for all $x\in \V$ and condition $(a)$ in the above proposition holds for all $c,u\in \V$. 
An equivalent definition is given in the following expanded form.

\begin{definition} \label{ftvn}
{\it A Fan-Theobald-von Neumann system (FTvN system, for short) is a triple $(\V,\W,\lambda)$, where 
$\V$ and $\W$ are  real inner product spaces and $\lambda:\V\rightarrow \W$ is a map satisfying the following conditions:
\begin{itemize}
\item [$(A1)$] $||\lambda(x)||=||x||$ for all $x\in \V$.
\item [$(A2)$] $\langle x,y\rangle \leq \langle \lambda(x),\lambda(y)\rangle$ for all $x,y\in \V$.
\item [$(A3)$] For any $c\in \V$ and $q\in \lambda(\V)$, there exists $x\in \V$ such that
\begin{equation}\label{A3}
\lambda(x)=q\quad\mbox{and}\quad \langle c,x\rangle=\langle \lambda(c),\lambda(x)\rangle.
\end{equation}
\end{itemize}
}
\end{definition}

\gap

A simple  example of a  FTvN system is $(\V,\R,\lambda)$, where 
$\V$ is a real inner product space and $\lambda(x):=||x||$ for all $x\in \V$. Here, $(A2)$ is just the Cauchy-Schwarz inequality and  condition $(A3)$,
 for $c\neq 0$, is seen by taking $x=(\frac{q}{||c||})c.$ 
Another  simple example is $(\V,\V,A)$, where $\V$ is a real inner product space and $A$ is a linear isometry on $\V$ (which is not necessarily invertible).
Also, if $(\V,\W,\lambda)$ is  a FTvN system, then so is 
$(\V,\W,A\circ\lambda)$, where $A:\W\rightarrow \W$ is a linear isometry.
Cartesian product of a finite number of  FTvN systems can be made into an FTvN system in an obvious way (by considering the sum of inner products and creating a $\lambda$ in a componentwise manner). 
We remark that conditions $(A1)$ and $(A2)$ need not imply $(A3)$; see Section 4 for an example.

\gap

In the next several results, we state some basic properties that hold in a FTvN  system. Some of these are elementary, 
and some proofs are modeled after similar ones existing in the literature \cite{lewis,bauschke et al}.

\gap

In a FTvN system $(\V,\W,\lambda)$, for any $c\in \V$, we define
$$\wl(c):=-\lambda(-c).$$

\begin{proposition}\label{fts-basic properties}
{\it In a FTvN system $(\V,\W,\lambda)$, the following statements hold for all $x,y,c,u\in \V$: 
\begin{itemize}
\item [$(a)$] $\lambda(\alpha\,x)=\alpha\,\lambda(x)$ for all $\alpha\geq 0$.
\item [$(b)$] $\langle\, \wl(c),\lambda(x)\rangle\leq \langle c,x\rangle \leq \langle \lambda(c),\lambda(x)\rangle$.
\item [$(c)$] $||\lambda(c)-\lambda(x)||\leq ||c-x||\leq ||\,\wl(c)-\lambda(x)||$.
\item [$(d)$] $\min\Big \{\langle c,x\rangle:\,x\in [u]\Big \}=\langle \,\wl(c),\lambda(u)\rangle.$
\item [$(e)$] $\max\Big \{||c-x||:x\in [u]\Big \}=||\wl(c)-\lambda(u)||.$
\end{itemize}
}
\end{proposition}

\begin{proof} We will use conditions $(A1)-(A3)$ in Definition \ref{ftvn}. \\
$(a)$  Let $\alpha\geq 0$. Using $(A1)$ and $(A2)$, we have  
$$||\lambda (\alpha\,x)-\alpha\lambda(x)||^2=||\lambda (\alpha\,x)||^2+\alpha^2||\lambda(x)||^2-2\alpha\langle \lambda(\alpha x),\lambda(x)\rangle \leq ||\alpha\,x||^2+\alpha^2||x||^2-2\alpha\langle \alpha x,x\rangle=0,$$
leading to the given statement.\\ 
\noindent $(b)$ The first inequality is obtained from $(A2)$ by putting $y=-c$. The second  inequality is just $(A2)$ with $y=c$. \\
$(c)$ Since 
$||\lambda(c)-\lambda(x)||^2-||c-x||^2=
2[\langle c,x\rangle-\langle \lambda(c),\lambda(x)\rangle]$ and $||\,\wl(c)-\lambda(x)||^2-||c-x||^2=
2[\langle c, x\rangle-\langle\, \wl(c),\lambda(x)\rangle]$, the inequalities in $(c)$ follow from Item $(b)$. 
\\
$(d)$ 
This is seen by  replacing  $c$ by $-c$ in Item $(a)$ of Proposition \ref{equivalence of linear and distance functions}, which holds because it is equivalent to $(A2)$ and $(A3)$.\\
$(e)$ Let $x\in [u]$. We have, as in the proof of Proposition \ref{equivalence of linear and distance functions},
$$||c-x||^2=||c||^2+||x||^2-2\langle c,x\rangle=||\,\wl(c)||^2+||\lambda(u)||^2-2\langle c,x\rangle.$$
Then, as $x$ varies over $[u]$, we have
$$\max\, ||c-x||^2=||\,\wl(c)||^2+||\lambda(u)||^2-2\min\,\langle c,x\rangle$$ which, by Item $(d)$, equals 
$||\,\wl(c)||^2+||\lambda(u)||^2-2\langle \,\wl(c),\lambda(u)\rangle=||\,\wl(c)-\lambda(u)||^2.$
\end{proof}

\gap

Note: All linear functions from $\V$ to $R$ considered in this paper are of the form $x\mapsto \langle c,x\rangle$ for some $c\in \V$. (These are continuous linear functionals on $\V$ and when $\V$ is a Hilbert space, every continuous linear functional arises this way.) As the map $\lambda$ is Lipschitz continuous (see Item $(c)$ in the above proposition) and norm preserving, every $\lambda$-orbit is closed in $\V$ and
lies on a sphere with origin as the center; it is compact
when $\V$ is finite dimensional.
{\it Throughout this paper, depending on the context, we use the same notation to denote an optimization problem as well as
 its optimal value.} For example, $\sup_{E}\,f$ denotes the problem of finding/describing the supremum of $f$ over the set $E$ as well as the supremum value.

\gap

In the setting of a FTvN system, the  following statements are simple consequences of convexity/concavity of a linear function and convexity of a distance function; they are based on  Propositions \ref{equivalence of linear and distance functions} and
 \ref{fts-basic properties}. Here, for a set $S$, we let $\conv(S)$ denote the convex hull of $S$.
\begin{itemize}
\item [$\bullet$]
$\max\Big \{\langle c,x\rangle: x\in \conv([u])\Big \}=\max\Big \{\langle c,x\rangle: \in [u]\Big\}=\langle \lambda(c),\lambda(u)\rangle$ and
\item [$\bullet$]
$\min\Big \{\langle c,x\rangle: x\in \conv([u])\Big \}=\min\Big \{\langle c,x\rangle: \in [u]\Big\}=\langle \,\wl(c),\lambda(u)\rangle.$
\item [$\bullet$]
$\max\Big \{||c-x||: x\in \conv([u])\Big \}=\max\Big \{||c-x||: x\in [u]\Big\}=||\,\wl(c)-\lambda(u)||.$
\end{itemize}

\gap

\begin{theorem} (Sublinearity theorem) \label{sublinearity theorem}
{\it Let $(\V,\W,\lambda)$ be a FTvN system. Then, for any $w\in \lambda(\V)$, the function $x\mapsto \langle w,\lambda(x)\rangle$ is sublinear, that is, for all $c,x,y\in \V$, we have:
\begin{equation}\label{sublinear} 
\langle \lambda(c),\lambda(x+y)\rangle \leq \langle \lambda(c),\lambda(x)\rangle+\langle \lambda(c),\lambda(y)\rangle.
\end{equation}
Consequently, $||\lambda(x+y)||\leq ||\lambda(x)+\lambda(y)||$.
}
\end{theorem}

\begin{proof} 
Fix $c,x,y\in \V$. For any $z\in [c]$, we have $\lambda(z)=\lambda(c)$ and 
$$\langle z,x+y\rangle =\langle z,x\rangle+\langle z,y\rangle\leq \langle \lambda(z),\lambda(x)\rangle+\langle\lambda(z),\lambda(y)\rangle=\langle \lambda(c),\lambda(x)\rangle+\langle \lambda(c),\lambda(y)\rangle.$$
Taking the maximum over $z$ and noting  
$\max\Big \{\langle z,x+y\rangle:\,z\in [c]\Big \}=\langle \lambda(c),\lambda(x+y)\rangle$
we get (\ref{sublinear}).
\\Now, letting $c=x+y$ in (\ref{sublinear}), we have 
\begin{equation}\label{CS}
||\lambda(x+y)||^2=\langle \lambda(x+y),\lambda(x+y)\rangle \leq \langle \lambda(x+y),\lambda(x)+\lambda(y)\rangle\leq 
||\lambda(x+y)||\,||\lambda(x)+\lambda(y)||,
\end{equation}
leading to  $||\lambda(x+y)||\leq ||\lambda(x)+\lambda(y)||$.
\end{proof}

\gap

The concept of commutativity, defined below, is central to the study of FTvN systems. 

\begin{definition}\label{commute}
{\it In a FTvN system $(\V,\W,\lambda)$ we say that 
elements $x,y\in \V$ commute if 
$$\langle x,y\rangle = \langle \lambda(x),\lambda(y)\rangle.$$
}
\end{definition}

As we shall see later, this concept, specialized to Euclidean Jordan algebras, is related to (in fact, stronger than) operator commutativity, which, in the settings of $\Sn$ and $\Hn$ reduces to the commutativity of two matrices.
In the presence of $(A2)$, we  can now interpret condition $(A3)$ in Definition \ref{ftvn}:  Every element $c$ in $\V$ commutes with some element in any given $\lambda$-orbit. Alternatively, defining 
$$C(x):=\{y\in \V:\, y \,\,\mbox{commutes with}\,\,x\},$$
$(A3)$ says that $C(c)\cap [u]\neq \emptyset$ for all $c,u\in \V$. {
\it We also note that in the optimization problem
$\max\Big \{\langle c,x\rangle:\,x\in [u]\Big \}$ whose optimal value is $\langle \lambda(c),\lambda(u)\rangle$,
the objective function attains its maximum at an $x^*$ if and only if $x^*$ commutes with $c$.}
Because of this, the concept of commutativity can be viewed as (part of) an  optimality condition. This may explain why
commutativity comes up in various optimization and variational inequality settings.

\gap
The following result describes commutativity in alternate ways.

\begin{proposition}\label{commutativity proposition}
{\it In a FTvN system $(\V,\W,\lambda)$, the following are equivalent:
\begin{itemize}
\item [$(a)$] $x$ and $y$ commute, that is, $\langle x,y\rangle=\langle \lambda(x),\lambda(y)\rangle.$
\item [$(b)$] $||\lambda(x)-\lambda(y)||=||x-y||.$
\item [$(c)$] $||\lambda(x+y)||= ||\lambda(x)+\lambda(y)||$.
\item [$(d)$] $\lambda(x+y)=\lambda(x)+\lambda(y).$
\end{itemize}
}
\end{proposition}

\begin{proof}
Using $(A1)$, we get the equalities  
$$||\lambda(x)-\lambda(y)||^2-||x-y||^2=
2[\langle x,y\rangle-\langle \lambda(x),\lambda(y)\rangle]=||\lambda(x+y)||^2- ||\lambda(x)+\lambda(y)||^2.$$
The equivalence $(a)\Leftrightarrow (b)\Leftrightarrow (c)$ follows.\\
$(c)\Rightarrow (d)$: When $(c)$ holds, we have the equality in the Cauchy-Schwarz inequality (\ref{CS}). Hence, 
one of the vectors in the (in)equality is  a nonnegative multiple of the other vector. Since $||x+y||=||\lambda(x+y)||= ||\lambda(x)+\lambda(y)||$, $(d)$ follows. \\
Finally, $(d)\Rightarrow (c)\Rightarrow (a)$.
\end{proof}

\gap

Arguments similar to the above will show that 
\begin{equation}\label{x and -c commute}
x\, \,\mbox{{\it commutes with}}\,\,-c \Longleftrightarrow \langle c,x\rangle=\langle \,\wl(c),\lambda(x)\rangle\Longleftrightarrow ||c-x||=||\,\wl(c)-\lambda(x)||.
\end{equation}

For ease of reference, we collect various statements equivalent to $(A3)$.

\begin{proposition}\label{items equivalent to A3}
{\it In a  FTvN system, the following hold:  
\begin{itemize}
\item [$(a)$] For any $c\in \V$ and $q\in \lambda(\V)$, there exists $x\in \V$ such that
$\lambda(x)=q$ and satisfying  one/all of the following conditions:
$$\langle c,x\rangle=\langle \,\lambda(c),\lambda(x)\rangle,\,\,\, 
||c-x||=||\lambda(c)-\lambda(x)||,\,\,\,\lambda(c+x)=\lambda(c)+\lambda(x).$$ 
\item [$(b)$]  
For any $c\in \V$ and $q\in \lambda(\V)$, there exists $x\in \V$ such that
$\lambda(x)=q$ and satisfying  one/both of the following conditions: 
$$\langle c,x\rangle=\langle \,\wl(c),\lambda(x)\rangle,\,\,\,||c-x||=||\,\wl(c)-\lambda(x)||.$$
\end{itemize}
}
\end{proposition}

\begin{proof}
$(a)$ This follows from $(A3)$ and Proposition \ref{commutativity proposition}.\\
$(b)$ We replace $c$ in $(A3)$ by $-c$ and use (\ref{x and -c commute}).
\end{proof}

A simple consequence of the above result is the following.

\begin{corollary}
{\it If $(\V,\W,\lambda)$ is a FTvN system, then $\lambda(\V)$ is a convex cone in $\W$; It is closed when $\V$ is finite dimensional.}
\end{corollary}

\begin{proof}
In view of Item $(a)$ in Proposition \ref{fts-basic properties}, $\lambda(\V)$ is a cone. If $\lambda(u)$ and $\lambda(v)$ are two elements in $\lambda(\V)$, then, applying Item $(a)$ in the above proposition with $q=\lambda(u)$ and $c=v$, we get an $x\in \V$ such that $\lambda(x)=q=\lambda(u)$ and $\lambda(x+c)=\lambda(x)+\lambda(c)=\lambda(u)+\lambda(v)$. Hence, $\lambda(u)+\lambda(v)\in \lambda(\V)$. 
Thus, $\lambda(\V)$ is a convex cone. Finally, if $\V$ is finite dimensional, we can use $(A1)$ and the continuity of $\lambda$ to show that $\lambda(\V)$ is closed.
\end{proof}

\gap
Motivated by certain concepts in Euclidean Jordan algebras, we now introduce the following.

\begin{definition}\label{spectral set}
{\it Let $(\V,\W,\lambda)$ be a FTvN system.
\begin{itemize}
\item [$\bullet$] A set $E$ in $\V$ is called a spectral set if it is of the form $E=\lambda^{-1}(Q)$ for some
$Q\subseteq \W$.
\item [$\bullet$] A function $\Phi:\V\rightarrow \R$ is said to be a spectral function if it is of the form
$\Phi=\phi\circ \lambda$ for some $\phi:\W\rightarrow \R$.
\end{itemize}
}
\end{definition}

\gap

It is clear that a spectral set is a union of  $\lambda$-orbits. The following implication is an intrinsic test for spectrality:   
$$\left [x\in E,\,\lambda(x)=\lambda(y)\right ]\Rightarrow y\in E.$$
(Then, we can let $Q:=\lambda(E)$ so that $E=\lambda^{-1}(Q)$.)
Also, a (real valued) function on $\V$ is a spectral function if and only if it is a constant on any $\lambda$-orbit.

\gap

Previously, we listed some elementary examples.
In the FTvN system $(\V,\W,\lambda)$, where $\W=\R$ and $\lambda(x)=||x||$, two elements  commute if and only if  one of them is a nonnegative  (scalar) 
multiple of the other. Also, $\lambda$-orbits are spheres centered at the origin and spectral functions are radial. In the FTvN system $(\V,\V,S)$, where $S$ is a linear isometry, any two elements commute. In fact, due to Item $(a)$ in Proposition \ref{fts-basic properties} and Item $(d)$ in Proposition \ref{commutativity proposition}, every FTvN
system where any two elements commute arises this way. 
\\

In the subsequent sections, we will provide nontrivial examples of FTvN systems. In particular, we will show/see
 the following: 
\begin{itemize}
\item [$\bullet$] If $\V$ is a Euclidean Jordan algebra of rank $n$ carrying the trace inner product 
and $\lambda:\V\rightarrow \Rn$ denotes the eigenvalue map, then  the triple $(\V,\Rn,\lambda)$ becomes a 
FTvN system. In this setting, a set in $\V$ is a spectral set if it is of the form $\lambda^{-1}(Q)$ for some (permutation invariant) set $Q$ in $\Rn$; a function $\Phi:\V\rightarrow \R$ is a spectral function if it is of the form 
$\phi\circ \lambda$ for some (permutation invariant) function $\phi:\Rn\rightarrow \R$. When $\V$ is simple, these are precisely sets and functions that are invariant under automorphisms of $\V$. Commutativity of elements $x$ and $y$ in the FTvN system $(\V,\W,\lambda)$ means that there is a 
Jordan frame ${\cal E}=\{e_1,e_2,\ldots, e_n\}$ in $\V$ such that  $x$ and $y$ have simultaneous order diagonal 
decomposition with respect to ${\cal E}$, that is, $x=\lambda_1(x)e_1+\lambda_2(x)e_2+\cdots+\lambda_n(x)e_n$ and 
$y=\lambda_1(y)e_1+\lambda_2(y)e_2+\cdots+\lambda_n(y)e_n$. This will be  referred to as the 
{\it strong operator commutativity} in the algebra $\V$. The algebras of $n\times n$ real/complex Hermitian matrices are primary examples of Euclidean Jordan algebras of rank $n$.  
\item [$\bullet$] If $\V$ is a finite dimensional real vector space and $p$ is a real homogeneous polynomial of degree $n$, hyperbolic with respect to a vector $e\in \V$, and  additionally complete and isometric \cite{bauschke et al}, then $(\V,\Rn,\lambda)$ becomes a FTvN system, where $\lambda(x)$ denotes the vector of roots of the univariate polynomial $t\rightarrow p(te-x)$ written in the decreasing order. In this setting, elements $x$ and $y$ commute if and only if 
$\lambda(x+y)=\lambda(x)+\lambda(y)$ (which is part of the definition of $p$ being `isometric').
\item [$\bullet$] If $(\V,\G,\gamma)$ is a normal decomposition system, then with $\W=\gamma(\V)-\gamma(V)$ and 
$\lambda=\gamma$, the triple $(\V,\W,\lambda)$ becomes a FTvN system. Here $\V$ is a  real inner product space and $\G$ is a closed subgroup of the orthogonal group of $\V$. Spectral sets and (real valued) functions are those that are invariant under elements of $\G$. In this setting, $x$ and $y$ commute in $(\V,\W,\lambda)$ if and only if there exists $A\in \G$ such that $x=A\gamma(x)$ and $y=A\gamma(y)$. The space of all $n\times n$ complex matrices is a primary example of a normal decomposition system. 
\item [$\bullet$] If $(\V,\G,F)$ is an Eaton triple, then with $\W:=F-F$ and $\lambda(x)$ denoting the unique element in 
$Orb(x)\cap F$, the triple $(\V,\W,\lambda)$ becomes a FTvN system. Here, $\V$ is a finite dimensional real inner product space, $\G$ is a closed subgroup of the orthogonal group of $\V$ and $F$ is a closed convex cone in $\V$.
(It is known that every Eaton triple is a normal decomposition system.)
\end{itemize}

\gap

We end this section with a remark about the `completion' of a FTvN system. Given a FTvN system $(\V,\W,\lambda)$, let $\overline{\V}$ and $\overline{\W}$ be 
the completions of the inner product spaces $\V$ and $\W$ respectively. Since $\lambda:\V\rightarrow \W$ is Lipschitz (see Proposition \ref{fts-basic properties}), there is a unique extension $\overline{\lambda}:\overline{\V}\rightarrow \overline{\W}$. Using elementary arguments and the Eberlein-Smulian Theorem \cite{megginson} (that in a Banach space, weak compactness is the same as weak sequential compactness), one can show that $(\overline{\V},\overline{\W},\overline{\lambda})$ is a FTvN system.

%
\section{Equivalent formulations of certain optimization problems over spectral sets}
Throughout this section, we consider  a FTvN system $(\V,\W,\lambda)$;  let 
$E$ be a spectral set in $\V$ and $\Phi$ be a spectral function on $\V$ with $\Phi=\phi\circ \lambda$ for some $\phi:\W\rightarrow \R$. Our goal is to reformulate an optimization problem over $E$ as a problem over $\lambda(E)$. In this section, we present several results dealing with combinations of linear/distance/convex functions and spectral functions. We start with an elementary result.

\begin{proposition}\label{elementary proposition}
{\it 
Suppose $A$ and $B$ are two sets in $\R$ with $B\subseteq A$. Then the following statements hold:
\begin{itemize}
\item [$(i)$] If every element of $A$ is less than or equal to some element of $B$, then $\sup\,A=\sup\,B.$ In this setting, attainment of one supremum implies that of the other. Moreover, if $\sup\,A$ is attained at $\overline{a}$, then $\overline{a}\in B$ and $\sup\,B$ is also attained at $\overline{a}$.
\item [$(ii)$] If every element of $A$ is greater than or equal to some element of $B$, then $\inf\,A=\inf\,B.$ In this setting, attainment of one infimum implies that of the other. Moreover, if $\inf\,A$ is attained at $\overline{a}$, then $\overline{a}\in B$ and $\inf\,B$ is also attained at $\overline{a}$.
\end{itemize}
}
\end{proposition} 

\begin{proof}
$(i)$ The inclusion $B\subseteq A$ implies that (in the extended real number system) $\beta:=\sup\,B\leq \sup\,A=:\alpha$. 
On the other hand, for any $a\in A$, there is a $b\in B$ such that $a\leq b\leq \beta$. This implies that 
$\alpha\leq \beta$. Hence, $\alpha=\beta$. Now suppose $\alpha$ is attained at $\overline{a}\in A$. Then, there 
is a $\overline{b}\in B$ such that $\overline{a}\leq \overline{b}$. But then, 
$\alpha=\overline{a}\leq\overline{b}\leq \beta=\alpha$ showing  $\overline{a}=\overline{b}\in B$ and 
$\overline{b}=\beta$. Finally, if $\sup\,B$ is attained at $b_0$, then $b_0\in A$ (recall $B\subseteq A$) and 
$b_0=\beta=\alpha$; thus, $\sup\,A$ is attained at $b_0$.\\
$(ii)$ The proof is similar to that of $(i)$.
\end{proof}

A simple example that illustrates Item $(i)$ above is: $A$ is the interval $(0,1)$ in $\R$ and $B$ is the set of all rationals in $A$.

\subsection{Optimizing a combination of a linear function and a spectral function over a spectral set}
We fix a $c\in \V$ and   define, for $x\in \V$ and $w\in \W$,
$$f(x):=\langle c,x\rangle,\,\,\, f^*(w):=\langle \lambda(c),w\rangle,\,\mbox{and}\,\,f_*(w):=\langle \,\wl(c),w\rangle.$$
In view of Item $(b)$ in Proposition \ref{fts-basic properties}, we have
\begin{equation}\label{linear sandwich inequalities}
  f_*(\lambda(x))\leq f(x)\leq f^*(\lambda(x)).
\end{equation}
We show below that for any spectral function $\Phi$ and any spectral set $E$, $\sup_{E}\,(f+\Phi) =\sup_{\lambda(E)}\,(f^*+\phi)$ and $\inf_{E}\,(f+\Phi) =\inf_{\lambda(E)}\,(f_*+\phi)$, with attainments leading to commutativity relations.
We derive this as a special case of a broader result dealing with a certain combination of $f$ and $\Phi$ instead of the sum of $f$ and $\Phi$. The motivation to consider such  an extension comes from the  work of Niezgoda \cite{niezgoda-commutation}.
\\

Given intervals $I$ and $J$ in $\R$, we  say that a function $L:I\times J\rightarrow \R$ is strictly increasing in the first variable if for each fixed $s^*\in J$, the function $t\rightarrow L(t,s^*)$ is strictly increasing over $I$.  Two simple examples are: $L(t,s)=t+s$ on $\R\times \R$ and $L(t,s)=ts$ on $\R\times (0,\infty)$.
The definition of $L$ increasing in the second variable is similar.

\begin{theorem}\label{linear L theorem}
{\it
Consider $f$, $f^*$, $f_*$, $\Phi$, and $E$ as given above.
Suppose the real valued function $L$ (defined on a product of two appropriate intervals in $\R$) is
strictly increasing in the first variable. Then, the following statements hold:
\begin{itemize}
\item [$(i)$]
$$\sup_{E} L(f,\Phi)=\sup_{\lambda(E)} L(f^*,\phi).$$
Also,  attainment of the supremum in one  problem implies that in the other. Moreover, if the supremum
of the problem  on the left is attained at $\barx\in E$, then $\barx$ commutes with $c$ in $(\V,\W,\lambda)$ and the maximum value is given by $L\Big (f^*(\lambda(\barx)),\phi(\lambda(\barx))\Big )$.
\item [$(ii)$]
$$\inf_{E} L(f,\Phi)=\inf_{\lambda(E)} L(f_*,\phi).$$
Also,  attainment of the infimum in one  problem implies that in the other. Moreover, if the infimum
of the problem  on the left is attained at $\barx\in E$, then $\barx$ commutes with $-c$ in $(\V,\W,\lambda)$ and the minimum value is given by $L\Big (f_*(\lambda(\barx)),\phi(\lambda(\barx)\Big )$.
\end{itemize}
 }
\end{theorem}

\noindent We note that the first variable in $L$ varies over an interval that contains the sets $f(E)$ and $f^*(\lambda(E))$ and the second variable varies over an interval that contains $\Phi(E)$. 
Also, we write $\sup_{E} L(f,\Phi)$  an abbreviation of $\sup_{x\in E} L(f(x),\Phi(x))$, etc.
\\

\begin{proof} $(i)$ Consider the following sets in $\R$:
$$A:=\Big \{ L\Big (f(x),\Phi(x)\Big ):x\in E\Big \}\quad\mbox{and}\quad 
B:=\Big \{ L\Big (f^*(q),\phi(q)\Big ):q\in \lambda(E)\Big \}.$$
Because $L$ is increasing in the first variable and (\ref{linear sandwich inequalities}) holds for any $x\in E$, we see that
every element in $A$ is less than or equal to some element of $B$. Also, from $(A3)$, for any $q\in \lambda(E)$, there is an $x\in \V$ such that $\lambda(x)=q$ and $f(x)=\langle c,x\rangle =\langle \lambda(c),\lambda(q)\rangle=f^*(q)$. (As $E$ is a spectral set, $x\in E$.) This shows that $B\subseteq A$. From Item $(i)$ in Proposition \ref{elementary proposition}, 
$\sup_{E}L(f,\Phi)=\sup\,A=\sup\,B=\sup_{\lambda(E)}L(f^*,\phi)$. Moreover, attainment in one problem implies that in the other. Now suppose that $\sup_{E}L(f,\Phi)$ is attained at $\barx\in E$. Then, with $\lambda(\barx)=\overline{q}$, we have
$$\sup_{E}L(f,\Phi)=L\Big (f(\barx),\Phi(\barx)\Big )\leq L\Big (f^*(\overline{q}),\phi(\overline{q})\Big )\leq \sup_{\lambda(E)}L(f^*,\phi)=
\sup_{E}L(f,\Phi),$$
where the first inequality comes from (\ref{linear sandwich inequalities}) and the assumed property of $L$. 
It follows that 
$$L\Big (f(\barx),\Phi(\barx)\Big )= L\Big (f^*(\overline{q}),
\phi(\overline{q})\Big ).$$
 Since $L$ is strictly increasing in the first variable and $\Phi(\barx)=\phi(\overline{q})$, we must have $f(\barx)=f^*(\lambda(\barx))$, that is $\langle c,\barx\rangle=\langle \lambda(c),\lambda(\barx)\rangle$, proving the commutativity of $c$ and $\barx.$ Clearly, the maximum value is given by 
$L\Big (f^*(\lambda(\barx)),\phi((\lambda(\barx))\Big )$.\\
$(ii)$ The proof is similar to that of $(i)$. Here we consider the sets 
$$A:=\Big \{ L\Big (f(x),\Phi(x)\Big ):x\in E\Big \}\quad\mbox{and}\quad
B:=\Big \{ L\Big (f_*(q),\phi(q)\Big ):q\in \lambda(E)\Big \}.$$
We use (\ref{linear sandwich inequalities}), Item $(b)$ in Proposition \ref{items equivalent to A3}, and Item $(ii)$ in Proposition \ref{elementary proposition} to get the equality of the two infimums and their attainment. The additional statement regarding the commutativity of $\barx$ and $-c$ comes from the equality 
$f(\barx)=f_*(\lambda(\barx))$.
\end{proof}

\gap

We specialize the above theorem by putting $L(t,s)=t+s$ on $\R\times \R$.

\begin{corollary} \label{linear sup-inf corollary} 
{\it In the setting of the above theorem, we have the equalities 
\begin{equation}\label{sup-inf equality for linear functions}
\sup_{E}\,(f+\Phi)=\sup_{\lambda(E)}\,(f^*+\phi)\quad\mbox{and}\quad \inf_{E}\,(f+\Phi)=\inf_{\lambda(E)}\,(f_*+\phi).
\end{equation}
Additionally, attainment in these lead to commutativity relations: $\barx$ commutes with $c$ in the supremum case and  
$\barx$ commutes with $-c$ in the infimum case. 
}
\end{corollary}

\gap

We end this section with an application to variational inequality problems and state one consequence.
Let $(\V,\W,\lambda)$ be a FTvN system,  $E$ be a set in $\V$, and
 $G:\V\rightarrow \V$ be an arbitrary map. Then, the {\it variational 
inequality problem} VI$(G,E)$ \cite{facchinei-pang}
is to find an $a\in E$ such that
$$\langle G(a),x-a\rangle \geq 0\,\,\mbox{for all}\,\, x\in E.$$

When $E$ is a closed convex cone in $\V$, VI$(G,E)$ becomes a {\it complementarity problem} \cite{facchinei-pang}:
Find $a\in \V$ such that
$$a\in E,\,G(a)\in E^*,\,\,\mbox{and}\,\,\langle a,G(a)\rangle=0,$$
where $E^*$ is the dual of $E$ in $\V$  given by
$E^*:=\{x\in \V:\langle x,y\rangle \geq 0\,\,\forall\,y\in E\}.$
\\

Corollary \ref{linear sup-inf corollary} leads to the following commutation principle which can be regarded as a generalization and an improvement of Proposition 8 in \cite{ramirez et al}.

\begin{theorem}\label{vi result}
{\it Suppose $E$ is a spectral set and
 $a$ solves \mbox{VI}$(G,E)$.   Then $a$ and $-G(a)$ commute in the given FTvN system.
}
\end{theorem}

\begin{proof}
If $a$ solves VI$(G,E)$, then $\langle G(a),x\rangle \geq \langle G(a),a\rangle$ for all $x\in E$. Hence, with $c:=G(a)$,
we see that $a$ is a minimizer of the problem $\min\{\langle c,x\rangle:x\in E\}$.
By Corollary \ref{linear sup-inf corollary} with $\phi=0$, $a$ commutes with $-c$ (which is $-G(a)$).
\end{proof}

\gap

We now state a result that is similar to (actually generalizes) Theorem 1.3 in \cite{gowda-jeong}.

\begin{theorem}\label{differentiable opt problem in ftvn}
{\it Let  $(\V,\W,\lambda)$ be a FTvN system where $\V$ is a Hilbert space,   $E$ be a convex spectral set, and $\Phi$ be a convex spectral function.
Further, let $L$ (defined on a product of appropriate intervals in $\R$) be strictly increasing in the first variable and increasing in the second variable. Suppose $h:\V\rightarrow \R$ is   Fr\'{e}chet differentiable
and  $a$ is a local minimizer of the problem $\min_{E}L\Big (h,\Phi\Big )$.
Then $a$ and $-h^{\prime}(a)$  commute in  $(\V,\W,\lambda)$.}
\end{theorem}

\noindent Note: Because $\V$ is a Hilbert space, by the Riesz representation theorem, we can regard the continuous linear functional
$h^\prime(a)$ as an element of $\V$.
\\

\begin{proof} 
Take any $x\in [a]$. Since $E$ is spectral and convex, for any $0\leq t\leq 1$, $y:=(1-t)a+tx\in E$. As $a$ is a local minimizer,   
for all positive $t$ near zero, we have
$$L\Big (h(a),\Phi(a)\Big )\leq L\Big (h(y),\Phi (y)\Big ).$$ Fix such a $t$.
As $\Phi$ is convex and spectral, 
$ \Phi (y)\leq (1-t)\Phi(a)+t\Phi(x)= (1-t)\Phi(a)+t\Phi(a)=\Phi(a).$ Since $L$ is increasing in the second variable, we have 
$$L\Big (h(y),\Phi (y)\Big )\leq L\Big (h(y),\Phi (a)\Big ).$$
Thus,
$$L\Big (h(a),\Phi(a)\Big )\leq L\Big (h(y),\Phi (a)\Big ).$$
Since $L$ is strictly increasing in the first variable,  
$$h(a)\leq h(y)=h\Big ( (1-t)a+tx\Big).$$
As this holds for all positive $t$ near zero, 
it follows that $\langle h^\prime(a),x-a\rangle \geq 0$ for all $x\in [a]$. 
So, $a$ solves VI$(h^\prime,[a])$.
By the previous result, $a$ and $-h^{\prime}(a)$ commute in  $(\V,\W,\lambda)$.
\end{proof}

We highlight one special case by taking $L(t,s)=t+s$ and $\Phi=0$.

\begin{corollary}
{\it Let  $(\V,\W,\lambda)$ be a FTvN system where $\V$ is a Hilbert space and    $E$ be a convex spectral set.
Suppose $h:\V\rightarrow \R$ is   Fr\'{e}chet differentiable
and  $a$ is a local minimizer of the problem $\min_{E}\,h$.
Then $a$ and $-h^{\prime}(a)$  commute in  $(\V,\W,\lambda)$.}
\end{corollary}

\subsection{Optimizing a combination of a distance function and a spectral function over a spectral set} 
We fix  $c\in \V$ and define, for $x\in \V$ and $w\in \W$,
$$g(x):=||c-x||,\,\,\, g^*(w):=||\,\wl(c)-w||,\,\,\mbox{and}\,\,g_*(w):=||\lambda(c)-w||.$$
In view of Item $(c)$ in Proposition \ref{fts-basic properties}, we have
\begin{equation}\label{distance sandwich inequalities}
g_*\Big ( \lambda(x)\Big )\leq g(x)\leq g^*\Big (\lambda(x)\Big ).
\end{equation}

Analogous to Theorem \ref{linear L theorem} we have the following. 

\begin{theorem}\label{distance L theorem}
{\it
Consider $g$, $g^*$, $g_*$, $\Phi$, and $E$ as given above.
Suppose the real valued function $L$ (defined on a product of two appropriate intervals in $\R$) is
strictly increasing in the first variable. Then, the following statements hold:
\begin{itemize}
\item [$(i)$]
$$\sup_{E} L(g,\Phi)=\sup_{\lambda(E)} L(g^*,\phi).$$
Also,  attainment of the supremum in one  problem implies that in the other. Moreover, if the supremum
of the problem  on the left is attained at $\barx\in E$, then $\barx$ commutes with $-c$ in $(\V,\W,\lambda)$ and the maximum value is given by $L\Big (g^*(\lambda(\barx)),\phi(\lambda(\barx))\Big )$.
\item [$(ii)$]
$$\inf_{E} L(g,\Phi)=\inf_{\lambda(E)} L(g_*,\phi).$$
Also,  attainment of the infimum in one  problem implies that in the other. Moreover, if the infimum
of the problem  on the left is attained at $\barx\in E$, then $\barx$ commutes with $c$ in $(\V,\W,\lambda)$ and the minimum value is given by $L\Big (g_*(\lambda(\barx)),\phi(\lambda(\barx)\Big )$.
\end{itemize}
 }
\end{theorem}

\gap

\begin{proof} $(i)$ The proof is similar to that of Item $(i)$ in Theorem \ref{linear L theorem}. We define sets $A$ and $B$ appropriately (by replacing $f$ by $g$),  
use (\ref{distance sandwich inequalities}), Item $(b)$ in Proposition \ref{items equivalent to A3},  and Item $(i)$ in Proposition \ref{elementary proposition} to get the equality of the two supremums. The attainment statement comes from the equality
$g(\barx)=g^*(\lambda(\barx))$, which, by (\ref{x and -c commute}) gives the commutativity of $\barx$ and $-c$.\\
$(ii)$ The proof is analogous to that of Item $(ii)$ in Theorem \ref{linear L theorem}. We replace $f$ by $g$,  
use (\ref{distance sandwich inequalities}), Item $(a)$ in Proposition \ref{items equivalent to A3}, and Item $(ii)$ in Proposition \ref{elementary proposition} to get the equality of the two infimums. The attainment statement comes from the equality $g(\barx)=g_*(\lambda(\barx))$, which gives the commutativity of $\barx$ and $c$. 
\end{proof} 

\noindent{\bf Remarks.} While comparing Theorems \ref{linear L theorem} and \ref{distance L theorem}, the reader will notice that in the supremum case (or the infimum case), commutativity statements are reversed: In the linear case, $\barx$ commutes with $c$ and in the distance case, $\barx$ commutes with $-c$. As we shall see in the next section this has to do with the 
both $f$ and $g$ being convex and derivatives of $f$ and $g^2$ at $\barx$ commuting with $\barx$. 
\\
 
We now specialize the above theorem by letting $L(t,s)=t+s$ on $\R\times \R$.

\begin{corollary} \label{distance sup-inf corollary}
{\it In the setting of the above theorem, we have the equalities
\begin{equation}\label{sup-inf equality for distance functions}
\sup_{E}\,(g+\Phi)=\sup_{\lambda(E)}\,(g^*+\phi)\quad\mbox{and}\quad \inf_{E}\,(g+\Phi)=\inf_{\lambda(E)}\,(g_*+\phi).
\end{equation}
Additionally, attainment in these lead to commutativity relations: $\barx$ commutes with $-c$ in the supremum case and  
$\barx$ commutes with $c$ in the infimum case. 
}
\end{corollary}

We end this section by describing  the distance between two spectral sets. Consider spectral sets $E$ and $F$ in $\V$. Then,
$$\inf\Big \{||x-y||:x\in E,y\in F\Big \}=\inf\Big \{||q-p||:q\in \lambda(E),p\in \lambda(F)\Big \}.
$$
This follows from (\ref{sup-inf equality for distance functions}) with $\phi=0$:
$$\inf_{x\in E,y\in F}||x-y||=\inf_{x\in E}\inf_{y\in F}||x-y||=\inf_{x\in E}\inf_{y\in F}||\lambda(x)-\lambda(y)||=
\inf_{q\in \lambda(E)}\inf_{p\in \lambda(F)}||q-p||.
$$
In a similar way, we have
$$\sup_{x\in E}\,\inf_{y\in F}||x-y||=\sup_{q\in \lambda(E)}\,\inf_{p\in \lambda(E)}||q-p||.$$
This leads to the equality of Hausdorff distances
\begin{equation} \label{hausdorff equality}
d_H(E,F)=d_H(\lambda(E),\lambda(F)),
\end{equation}
where $d_H(E,F)$ is given by
$$d_H(E,F):=\max\Big \{\sup_{x\in E}\,\inf_{y\in F}||x-y||,\,\sup_{y\in F}\,\inf_{x\in E}||y-x||\Big\},$$
etc.

\subsection{Optimizing a combination of a convex function and a spectral function over a  spectral set}
In the previous sections, we considered the problems dealing with  linear and distance functions. Noting that these functions are convex, one may raise the question of extending the results of the previous sections to
convex functions. In this section, we provide some  answers  in the setting of finite dimensional spaces.

Consider a FTvN system $(\V,\W,\lambda)$, where $\V$ is  finite dimensional.
Let $h:\V\rightarrow \R$ be a convex  function. Then, $h$ is continuous and
can be realized as the
supremum of affine functions (\cite{ekeland-temam}, page 13):  For some collection $\{(c_i,\alpha_i):i\in I\}$ in $\V\times \R$,
$$h(x)=\sup_{i\in I}\big [ \langle c_i,x\rangle+\alpha_i\big ]\quad (x\in \V).$$
Correspondingly, we define two  extended real valued  functions on $\W$:
$$h^*(w):=\sup_{i\in I}\big [ \langle \lambda(c_i),w\rangle+\alpha_i\big ]\quad (w\in \W)$$ and
$$h_*(w):=\sup_{i\in I}\big [\, \langle \,\wl(c_i),w\rangle+\alpha_i\big ]\quad (x\in \W).$$
These functions, as  supremums of affine functions, are convex (possibly, extended real valued).
Moreover, from Proposition \ref{fts-basic properties}, we have the inequalities
$$\langle\, \wl(c_i),\lambda(x)\rangle+\alpha_i\leq \langle c_i,x\rangle+\alpha_i \leq \langle \lambda(c_i),\lambda(x)\rangle+\alpha_i,$$ and consequently,
\begin{equation}\label{lower and upper convex functions}
h_*(\lambda(x))\leq h(x)\leq h^*(\lambda(x))\quad \mbox{for all}\,\,x\in \V.
\end{equation}
In the result below, we show  that {\it on $\lambda(\V)$,
$h^*$ is finite valued and is independent of the representation of $h$.}
But first, we describe a particular representation of $h$ based on subdifferentials.\\
For any $x^*\in \V$, let $\partial h(x^*)$ denote the subdifferential of $h$ at $x^*$ (which is nonempty, compact and convex). By definition, $c\in \partial h(x^*)$ means that $h(x)\geq h(x^*)+\langle c,x-x^*\rangle$ for all $x\in \V$. Then, letting $\alpha:=h(x^*)-\langle c,x^*\rangle$, we have
$h(x)\geq \langle c,x\rangle+\alpha$ for all $x\in \V$ with equality at $x^*$. This gives the representation  
\begin{equation} \label{subdifferential representation}
h(x)=\sup_{(c,\alpha)\in \Omega}\big [\langle c,x\rangle+\alpha\big ],
\end{equation}
where $$\Omega:=\Big \{(c,\alpha)\in \V\times \R: \,\,\mbox{for some}\,\,x^*\in \V,\,c\in \partial h(x^*)\,\,\mbox{and}\,\,\alpha=h(x^*)-\langle c,x^*\rangle\Big \}.$$

\gap

\begin{theorem} \label{convex theorem}
{\it Let $(\V,\W,\lambda)$ be a FTvN system where $\V$ is finite dimensional. Let $E$ be a spectral set in $\V$ and $\Phi$ be a spectral function on $\V$. 
Suppose $h:\V\rightarrow \R$ is convex. By fixing a representation of $h$, we
define the corresponding extended real valued convex function $h^*$ on $\W$.  Let $L$ (defined on the product of appropriate intervals in $\R$) be strictly increasing in the first variable.
Then we have the following:
\begin{itemize}
\item [$(a)$] For any $q\in  \lambda(\V)$, $h^*(q)=\max\Big \{h(x):\lambda(x)=q\Big \}<\infty$. Hence,  on $\lambda(\V)$, $h^*$ is independent of the representation of $h$.
\item [$(b)$] \hspace{1cm}$\sup_{E}L\Big (h,\Phi\Big )=\sup_{\lambda(E)} L\Big (h^*,\phi\Big ).$ \\
Also, attainment of supremum in one problem implies that in the other. Moreover, if the problem on the left is attained at $\overline{x}\in E$, then $\overline{x}$ commutes with every element in the subdifferential of $h$ at $\overline{x}$.
\item [$(c)$] $h_*(q)\leq \min \Big \{h(x):\lambda(x)=q\Big \} \quad \Big (q\in \lambda(\V)\Big )$.
\item [$(d)$] $\inf_{\lambda(E)}L\Big (h_*,\phi \Big )\leq \inf_{E}L\Big (h,\Phi\Big ).$
\end{itemize}
}
\end{theorem}

\begin{proof}
\noindent $(a)$ Let $q\in \lambda(\V)$.
Then, the set  $\{x\in V:\lambda(x)=q\}$ is nonempty, closed and bounded (as $\lambda$ is continuous and norm preserving)
in $\V$. As $\V$ is  finite dimensional, this set is compact in $\V$. By the continuity of  $h$ we see that
$\max\Big \{h(x):\lambda(x)=q\Big \}$ exists.  Now, the inequality $h(x)\leq h^*(\lambda(x))$ implies that
$$\max\Big \{h(x):\lambda(x)=q\Big \}\leq h^*(q).$$
We recall the representations
$h(x)=\sup_{i\in I}\big [ \langle c_i,x\rangle+\alpha_i\big ]$ on $\V$
and $h^*(w):=\sup_{i\in I}\big [ \langle \lambda(c_i),w\rangle+\alpha_i\big ]$
on $\W$.
By $(A3)$ in Definition \ref{ftvn}, for every $i\in I$, there is an $x_i\in \V$ such that
$\lambda(x_i)=q$ and
$\langle \lambda(c_i),q\rangle=\langle c_i,x_i\rangle.$ Consequently,
$$\langle \lambda(c_i),q\rangle+\alpha_i\leq h(x_i)\leq \max\Big \{h(x):\lambda(x)=q\Big\}.$$
Taking the supremum over $i$, we get
$$h^*(q)\leq \max\Big \{h(x):\lambda(x)=q\Big \}.$$
We thus have the equality $h^*(q)=\max\Big \{h(x):\lambda(x)=q\Big \}$. This gives the 
 finiteness of $h^*(q)$ and shows that $h^*(q)$ depends on the values of $h$ alone and not on the representation of $h$. \\
$(b)$
Consider the two problems stated in Item $(b)$. Define the sets
$$A:=\Big \{ L\Big (h(x),\Phi(x)\Big ):x\in E\Big \}\quad\mbox{and}\quad
B:=\Big \{ L\Big (h^*(q),\phi(q)\Big ):q\in \lambda(E)\Big \}.$$
Because $L$ is increasing in the first variable and (\ref{lower and upper convex functions}) holds, we see that every element of $A$ is less than or equal to some element of $B$. Moreover, by Item $(a)$, for each $q\in \lambda(E)$, $h^*(q)=h(x)$ for some $x\in E$ with $\lambda(x)=q$. Thus $B\subseteq A$. 
We now use Proposition \ref{elementary proposition} to see  
$\sup\,A=\sup\,B$ which gives the equality 
$$\overline{\alpha}:=\sup_{E}L\Big (h,\Phi\Big )=\sup_{\lambda(E)} L\Big (h^*,\phi\Big )=:\overline{\beta}.$$ 
We also see the attainment of  one supremum implying that of the other. 
To see the commutativity part in $(b)$,  
suppose that $\sup_{E}L\Big (h,\Phi\Big )$ is attained at $\barx\in E$.
Let
$\overline{q}:=\lambda(\overline{x})\in \lambda(E)$. Then, by (\ref{lower and upper convex functions}),
$$\overline{\alpha}=L\Big (h(\overline{x}),\Phi(\overline{x})\Big )\leq L\Big (h^*(\overline{q}),\phi(\overline{q})\Big ) 
\leq \overline{\beta}=\overline{\alpha}.$$
Because $L$ is strictly increasing in the 
first variable and $\Phi(\overline{x})=\phi(\overline{q})$, we get
$$h(\overline{x})=h^*(\overline{q}).$$
Now, noting that on $\lambda(\V)$,  $h^*$ is independent of the representation of $h$, we consider the subdifferential representation (\ref{subdifferential representation}):
$$h(x)=\sup_{(c,\alpha)\in \Omega}\big [\langle c,x\rangle+\alpha\big ].$$
Then, for all $q\in \lambda(\V)$, 
$$h^*(q)=\sup_{(c,\alpha)\in \Omega}\big [ \langle \lambda(c),q\rangle+\alpha\big ].$$
Now, take any $c\in \partial h(\overline{x})$ and let $\alpha:=h(\overline{x})-\langle c,\overline{x}\rangle.$ Then, $(c,\alpha)\in \Omega$ and
$$h(\overline{x})=\langle c,\overline{x}\rangle+\alpha\leq \langle \lambda(c),\lambda(\overline{x})\rangle+\alpha\leq h^*(\overline{q})=h(\overline{x}).$$
Hence, $\langle c,\overline{x}\rangle =\langle \lambda(c),\lambda(\overline{x})\rangle$, proving the commutativity of $c$ and $\overline{x}$.\\
Finally, as  $L$ is increasing in the first variable, Items $(c)$ and $(d)$  are immediate from (\ref{lower and upper convex functions}).
\end{proof}

\gap

We highlight one special case.

\begin{corollary} \label{convex sup corollary}
{\it Suppose $(\V,\W,\lambda)$ is a FTvN system where $\V$ is finite dimensional. Let $E$ be a spectral set and $\Phi$ be a spectral function on $\V$.
Suppose  $h:\V\rightarrow \R$ is convex and $\barx$ is  an optimizer of the problem $\max_{E} (h+\Phi).$ Then, $\barx$ 
commutes with every element in the subdifferential of $h$ at $\overline{x}$.
}
\end{corollary}

\noindent{\bf Remarks.} 
Given the convex function $h$, the construction of $h_*$ is unsatisfactory for two reasons: First, when $h$ is the distance function $g$ considered in the previous section, $h_*$ may be different from $g_*$. Second, unlike for the linear and distance functions, the equality
$h_*(q)=\min\{h(x):\lambda(x)=q\}$ may not hold, see the example below. We can remedy this at this expense of losing convexity by defining:
$$h_{**}(q):=\min\{h(x):\lambda(x)=q\}\quad (q\in \lambda(\V)).$$
Then, one can show that $\inf_{E}L(h,\Phi)=\inf_{\lambda(E)}L(h_{**},\phi)$. We omit the details.
\\

\begin{example}
 Let $(\V,\W,\lambda)=(\R^2,\R^2,\lambda)$, where for any $q\in \R^2$, $\lambda(q)=q^\downarrow$ (the decreasing rearrangement of $q$, see Section 4.1 for details). Here, $\wl(q)=-[(-q)^\downarrow)]= q^\uparrow$ (the increasing rearrangement of $q$).
Now, let $c_1=(1,0)$ and $c_2=(-1,0)$ so that $c_1^\uparrow=(0,1)$ and $c_2^\uparrow=(-1,0)$. Consider the convex function
$$h(x):=\max\Big \{\langle c_1,x\rangle, \langle c_2,x\rangle\Big \}=|x_1|,$$
where $x=(x_1,x_2)\in \R^2$.
We have, for $w=(w_1,w_2)\in \R^2$,
$$h_*(w)=\max\Big \{\langle c_1^\uparrow,w\rangle, \langle c_2^\uparrow,w\rangle\Big \}=\max\Big\{w_2,-w_1\Big \}$$ and
$$h_{**}(q)=\min\{h(x):\lambda(x)=q\}=\min\{|q_1|,|q_2|\}\quad (q\in \lambda(\V)).$$
With $q=(1,-1)$, we have $\{x\in \R^2:\lambda(x)=q\}=\{(1,-1),(-1,1)\}$; so,
$\min\Big \{h(x):\lambda(x)=q\Big \}=1.$ However,
$h_*(q)=-1.$ This shows that the inequality in Item $(c)$ of the above theorem can be strict. 
We also note that $\min_{\R^2}h=0$, while $\inf_{\lambda(\R^2)}h_*=-\infty.$ We also observe that $h_{**}$ is nonconvex.  
\end{example}

\gap

We end this section by stating a result that is similar to Corollary \ref{convex sup corollary}, but dealing with the 
minimum of a convex function.

\begin{proposition}
{\it Suppose $(\V,\W,\lambda)$ is a FTvN system, where $\V$ is finite dimensional. Let 
 $E$ be a convex spectral set in  $\V$.
Suppose  $h:\V\rightarrow \R$ is convex and $\barx$ is  an optimizer of the problem $\min_{E} h.$ Then, $\barx$
commutes with $-c$ for some  element $c$ in the subdifferential of $h$ at $\overline{x}$.
}
\end{proposition}

\begin{proof} Let $\lchi$ denote the indicator function of $E$ (so it takes the value zero on $E$ and infinity outside of $E$). Then, $\barx$ is  a (global) optimizer of the problem $\min (h+\lchi)$ and so
$$0\in \partial\,(h+\lchi)(\overline{x})=\partial\,h(\overline{x})+\partial\,\lchi(\overline{x}),$$
where the equality comes from the subdifferential sum formula (\cite{rockafellar}, Theorem 23.8).
Hence, there is a $c\in \partial\,h(\overline{x})$ such that $-c\in \partial\,\lchi(\overline{x}).$ 
This $c$ will have the property that $$\langle c,x-\overline{x}\rangle \geq 0\,\,\mbox{for all}\,\,x\in E,$$ that is,
$\overline{x}$ is a minimizer of the problem $\min\{\langle c,x\rangle:x\in E\}$. By Corollary \ref{linear sup-inf corollary} with $\phi=0$, $\overline{x}$ commutes with $-c$.
 \end{proof}

\section{Euclidean Jordan algebras}
In this section, we show that every Euclidean Jordan algebra is a FTvN system and illustrate our previous results. We  start with some preliminaries. The Euclidean $n$-space $\Rn$ carries the usual inner product. 
For any $q\in \Rn$, we let $q^\downarrow$ denote the decreasing rearrangement of $q$ (that is,
$q^\downarrow_1\geq q^\downarrow_2\geq \cdots\geq q^\downarrow_n)$; for any $Q\subseteq \Rn$, we let $Q^\downarrow:=\{q^\downarrow:q\in Q\}$. The symbol  $\Sigma_n$ denotes the set of all permutation matrices on $\Rn$. For a set $S$ in $\Rn$, we write $\Sigma_n(S):=\{\sigma(s):\sigma\in \Sigma_n,\,s\in S\}$. We say that a set $Q$ in $\Rn$ is {\it permutation invariant} if $\sigma(Q)=Q$ for all $\sigma\in \Sigma_n$. (The word {\it symmetric} is also used in some literature.)\\

Let 
 $(\V,\circ, \langle\cdot,\cdot\rangle)$ denote a  {\it Euclidean Jordan algebra of rank $n$} \cite{faraut-koranyi}, 
where  $x\circ y$  and $\langle x, y\rangle$ denote, respectively,  
the Jordan product and inner product of two elements $x$ and $y$. It is known \cite{faraut-koranyi} 
that any Euclidean Jordan algebra is a direct product/sum
of simple Euclidean Jordan algebras and every simple Euclidean Jordan algebra is isomorphic one of five algebras,
three of which are the algebras of $n\times n$ real/complex/quaternion Hermitian matrices. The other two are: the algebra of $3\times 3$ octonion Hermitian matrices and the Jordan spin algebra.
By the spectral theorem \cite{faraut-koranyi}, every element $x$ in $\V$ has a decomposition
$x=q_1e_2+q_2e_2+\cdots+q_ne_n$, where $q_1,q_2,\ldots, q_n$ are  the
eigenvalues of $x$ and $\{e_1,e_2,\ldots, e_n\}$ is a Jordan frame. (The eigenvalues remain the same in any such representation.) Defining the sum of eigenvalues of $x$ as the trace of $x$, we note (the known fact) that the trace inner product $tr(x\circ y)$ is also compatible with the Jordan product. {\it Henceforth, we assume that the inner product in $\V$ is  this trace inner product}, that is,
$\langle x,y\rangle=tr(x\circ y).$
Working with this inner product allows us to say that each Jordan frame is orthonormal. 

For any $x\in \V$, we  let $\lambda(x):=(\lambda_1(x),\lambda_2(x),\ldots, \lambda_n(x))$ denote the vector of eigenvalues of $x$ written in the decreasing order. 
Then $\lambda:\V\rightarrow \Rn$ denotes the {\it eigenvalue map}. 
Given a Jordan frame ${\cal E}=\{e_1,e_2,\ldots, e_n\}$ in $\V$, we fix its enumeration/listing
and  define for any $q=(q_1,q_2,\ldots,q_n)\in \Rn$,
$$q*{\cal E}:=\sum_{i=1}^{n}q_ie_i.$$
We note that 
\begin{equation}\label{eigenvalue image}
\lambda(q*{\cal E})=q^\downarrow.
\end{equation}
A set $E$ in $\V$ is said to be a {\it spectral set} if it is of the form $\lambda^{-1}(Q)$ for some $Q\subseteq \Rn$. We state the following simple (easily verifiable) result.
\begin{proposition} \label{easy prop}
{\it Let $E=\lambda^{-1}(Q)$ for some $Q\subseteq \Rn$. Then the following statements hold:
\begin{itemize}
\item [$(i)$] $x\in E\Leftrightarrow \lambda(x)\in \lambda(E)$.
\item [$(ii)$] $\lambda(E)=Q\cap Q^\downarrow$.
\item [$(iii)$] $\lambda^{-1}(Q)=\lambda^{-1}(Q\cap Q^\downarrow)=\lambda^{-1}\Big (\Sigma_n(Q\cap Q^\downarrow)\Big )$.
\end{itemize}
}
\end{proposition}

Because of the third item above, {\it we can always write a spectral set as the $\lambda$-inverse image of permutation invariant set.}
A function $\Phi:\V\rightarrow \R$ is a spectral function if it is of the form $\Phi=\phi\,\circ \lambda$ for some $\phi:\Rn\rightarrow \R$. Note that we can always rewrite such a function as $\Phi=\phi_1\,\circ \lambda$, where $\phi_1:\Rn\rightarrow \R$ is permutation invariant, that is, $\phi_1(\sigma(q))=\phi_1(q)$ for all $\sigma\in \Sigma_n$ and $q\in \Rn$. 
In the case of a simple algebra, spectral sets and functions are precisely those that are invariant under automorphisms of $\V$ \cite{jeong-gowda-spectral set}. (An automorphism of $\V$ is a linear isomorphism of $\V$ that preserves the Jordan product.)\\

We say that elements $x$ and $y$ {\it operator commute} in $\V$ if there is a
Jordan frame ${\cal E}$ in $\V$ such that $$x=q*{\cal E}\quad\mbox{and}\quad y=p*{\cal E}$$
for some $q,p\in \Rn$. It is well-known that this is equivalent to the commutativity of the linear operators $L_x$ and $L_y$, where
$L_x(z)=x\circ z$, etc. We say that $x$ and $y$ {\it strongly operator commute}
(or said to be `simultaneously  order diagonalizable' \cite{lim et al} or said to have `similar joint decomposition' \cite{baes}) in $\V$ 
if there is a Jordan frame ${\cal E}$ such that
$$x=\lambda(x)*{\cal E}\quad\mbox{and}\quad y=\lambda(y)*{\cal E}.$$

The following result extends the Hardy-Littlewood-P\'{o}lya rearrangement inequality in $\Rn$ \cite{marshall-olkin} and Fan-Theobald trace inequality for real/complex Hermitian matrices \cite{fan,theobald} to (general) Euclidean Jordan algebras. One way of proving it to show that the result holds in a simple Euclidean Jordan algebra (see \cite{lim et al,gowda-tao}) and then use  the above mentioned Hardy-Littlewood-P\'{o}lya rearrangement inequality. For a direct proof, see \cite{baes}.

\begin{theorem}\label{baes}
{\it Let $\V$ be a Euclidean Jordan algebra carrying the trace inner product. Then, for  $x,y\in \V$, we have
$$\langle x,y\rangle \leq \langle \lambda(x),\lambda(y)\rangle$$ with equality if and only if $x$ and $y$ 
strongly operator commute.
}
\end{theorem}

\gap

This leads to the following result.

\begin{theorem}\label{eja is ftvn}
{\it Consider a Euclidean Jordan algebra of rank $n$ carrying trace inner product and let $\lambda:\V\rightarrow \Rn$ denote the eigenvalue map. Then,  the triple $(\V,\Rn,\lambda)$ is a FTvN system.
}
\end{theorem}

\begin{proof}
 We verify conditions $(A1)-(A3)$ in Definition \ref{ftvn}. For any $x$, consider the 
spectral decomposition $x=q_1e_2+q_2e_2+\cdots+q_ne_n$, where $q_1,q_2,\ldots, q_n$ are  the
eigenvalues of $x$ and $\{e_1,e_2,\ldots, e_n\}$ is a Jordan frame.  By our assumption that $\V$ carries the trace inner product, the Jordan frame is orthonormal. Hence $||x||^2=\sum_{i=1}^{n}|q_i|^2=||\lambda(x)||^2$. This verifies $(A1)$. Condition $(A2)$ follows from Theorem \ref{baes}. To see $(A3)$, let $c\in \V$ and $q\in \lambda(\V)$. We write  the spectral decomposition of $c$ as $c=\lambda(c)*{\cal E}$ for some Jordan frame ${\cal E}$. Now, as the components of $q$ are decreasing, letting $x:=q*{\cal E}$ we see that $\lambda(x)=q$.
Since ${\cal E}$ is orthonormal, $\langle c,x\rangle =\langle \lambda(c),\lambda(x)\rangle$ and so, $(A3)$ is verified.
\end{proof}

\gap

The following result shows that strong operator commutativity in the Euclidean Jordan algebra $\V$ is equivalent to commutativity in the FTvN system $(\V, \Rn,\lambda)$.

\begin{proposition}\label{strong operator commutativity in eja}
{\it For elements $x$ and $y$ in  a Euclidean Jordan algebra $\V$ with trace inner product, the following are equivalent:
\begin{itemize}
\item [$(i)$] $x$ and $y$ strongly operator commute in $\V$.
\item [$(ii)$] $\lambda(x+y)=\lambda(x)+\lambda(y)$.
\item [$(iii)$] $\langle x,y\rangle=\langle \lambda(x),\lambda(y)\rangle$, that is, $x$ and $y$ commute in the FTvN system $(\V, \Rn,\lambda)$.
\end{itemize}
}
\end{proposition}

\begin{proof}
If $(i)$ holds, then we can write 
$x=\lambda(x)*{\cal E}$ and $y=\lambda(y)*{\cal E}$ for some Jordan frame ${\cal E}$. Then, $x+y=[\lambda(x)+\lambda(y)]*{\cal E}$ and (as components of $\lambda(x)+\lambda(y)$ are in  decreasing order) $\lambda(x+y)=\lambda(x)+\lambda(y)$. 
Thus, $(i)\Rightarrow (ii)$. \\
As $(\V, \Rn,\lambda)$ is a FTvN system, the equivalence of $(ii)$ and $(iii)$ comes from Proposition \ref{commutativity proposition}.
\\Finally, the equivalence of $(iii)$ and $(i)$ follows from Theorem \ref{baes}.
\end{proof}

We now describe two consequences of Theorem \ref{eja is ftvn}. Thanks to the spectral theorem, it is easy to see that an element $u$ in $\V$ is   an {\it idempotent}, that is $u^2=u$, if and only if zero and one are the only possible eigenvalues of $u$. For any natural number $k$, $1\leq k\leq n$, consider an idempotent $u$
 having $k$ nonzero eigenvalues (in which case, we say that $u$ has rank $k$). Then $\lambda(u)=(1,1,\ldots, 1,0,0,\ldots, 0)$ in $\Rn$ and
the $\lambda$-orbit $[u]$ equals $\mathcal{J}^k(\V)$, the set of all idempotents of rank $k$ in $\V$. For such a $u$, the statement
$\max\{\langle c,x\rangle:x\in [u]\}=\langle \lambda(c),\lambda(u)\rangle$ {\it extends (\ref{moldovan result}) from simple Euclidean Jordan algebras to general Euclidean Jordan algebras.
} \\
For the second consequence,  consider (\ref{sublinear}) and replace $c$ by the above $u$ to get the inequality
\begin{equation}\label{k inequality}
\sum_{i=1}^{k}\lambda_i(x+y)\leq \sum_{i=1}^{k}\lambda_i(x) +\sum_{i=1}^{k}\lambda_i(y).
\end{equation}
When $k=n$,
\begin{equation}\label{n equality}
\sum_{i=1}^{n}\lambda_i(x+y)=\langle x+y,e\rangle=\langle x,e\rangle+\langle y,e\rangle=\sum_{i=1}^{n}\lambda_i(x)+\sum_{i=1}^{n}\lambda_i(y),
\end{equation}
where $e$ denotes the unit element in $\V$. These two statements together say (by definition) that $\lambda(x+y)$ is {\it majorized} by $\lambda(x)+\lambda(y)$ \cite{marshall-olkin}.  While such a statement is known for simple Euclidean Jordan algebras (\cite{moldovan}, Theorem 19) and the result for general algebras  can be proved by elementary means, for lack of explicit reference, we  record this fact below using the standard notation for majorization. As we shall see in the section on hyperbolic polynomials, this is a particular case of a far reaching generalization due to Gurvits. 

\begin{theorem}\label{majorization theorem}
{\it In any Euclidean Jordan algebra $\V$, for any two elements $x$ and $y$,
$$\lambda(x+y)\prec \lambda(x)+\lambda(y).$$
}\end{theorem}

\subsection{Some specialized results in Euclidean Jordan algebras}

Since every Euclidean Jordan algebra can now be regarded as a FTvN system, all the results of Section 3 could be 
stated for Euclidean Jordan algebras. 
Instead of repeating these, we now state some specialized results.

First consider the Euclidean Jordan algebra $\Rn$ (with the usual inner product and componentwise product). 
In this setting, for any $q\in \Rn$, $\lambda(q)=q^\downarrow.$ Also,  
as there is only one Jordan frame in $\Rn$ (up to permutation, namely, the 
standard coordinate basis),  any two elements in $\Rn$ operator commute, and {\it two vectors $p$ and $q$ in $\Rn$ strongly operator commute if and only if for some permutation matrix $\sigma$, $p=\sigma(p^\downarrow)$ and $q=\sigma(q^\downarrow).$} 
This simple observation will allow us to construct an example of an inner product space satisfying conditions $(A1)$ and $(A2)$ of Definition \ref{ftvn}, but not $(A3)$: In the Euclidean Jordan algebra (FTvN system) $\R^3$, consider the (sub)space $\Z$ generated by vectors $p=(3,2,1)$ and $q=(-1,0,0)$ and let $\mu$ denote the restriction of $\lambda$ to this subspace. Then, in the triple $(\Z,\Rn,\mu)$, conditions $(A1)$ and $(A2)$ hold but not $(A3)$. 
This example will also show that a subspace of an FTvN system need not be an FTvN system.\\

We now specialize  (\ref{sup-inf equality for linear functions}) for $\V=\Rn$, $\Phi=0$,
 and $f(x)=\langle p,x\rangle$ where $p\in \Rn$.
Let $Q\subseteq \Rn$ so that  $E:=\lambda^{-1}(Q)=\Sigma_n(Q\cap Q^\downarrow)$  and $\lambda(E)=Q\cap Q^\downarrow$. Then,
\begin{equation}\label{sup  in rn}
\sup\Big \{\langle p,q\rangle: q\in \Sigma_n(Q\cap Q^\downarrow)\Big \}=
\sup\Big \{\langle p^\downarrow,q\rangle: q\in Q\cap Q^\downarrow\Big \}.
\end{equation}
In particular, when $Q$ is permutation invariant, we have  $Q\cap Q^\downarrow =Q^\downarrow$ and 
$\Sigma_n(Q\cap Q^\downarrow)=Q$. So 
\begin{equation}\label{linear sup for pi set}
\sup\Big \{\langle p,q\rangle: q\in Q\Big \}=
\sup\Big \{\langle p^\downarrow,q\rangle: q\in Q^\downarrow\Big \}
\quad (Q\,\,\mbox{permutation invariant}).
\end{equation}
 
We can now combine (\ref{sup-inf equality for linear  functions})  with (\ref{sup in rn}) to get a statement in a Euclidean Jordan algebra $\V$: If $Q\subseteq \Rn$ and $E=\lambda^{-1}(Q)$ in $\V$, then for any $c\in \V$,
\begin{equation} \label{sup in eja}
\sup\Big \{\langle c,x\rangle: x\in \lambda^{-1}(Q)\Big \}=
\sup\Big \{\langle \lambda(c),q\rangle: q\in Q\cap Q^\downarrow\Big \};
\end{equation}
Moreover, when $Q$ is permutation invariant, thanks to (\ref{linear sup for pi set}),
\begin{equation}\label{three sups}
\sup\Big \{\langle c,x\rangle: x\in \lambda^{-1}(Q)\Big \}=
\sup\Big \{\langle \lambda(c),q\rangle: q\in Q^\downarrow\Big \}=\sup\Big \{\langle \lambda(c),q\rangle: q\in Q\Big \}.
\end{equation}
Analogous statements can be made for the infimum in place of the supremum.

Similarly, using (\ref{sup-inf equality for distance functions}), we have (in the setting of a Euclidean Jordan algebra $\V$),
\begin{equation} 
\inf\Big \{||c-x||: x\in \lambda^{-1}(Q)\Big \}=
\inf\Big \{||\lambda(c)-q||: q\in Q\cap Q^\downarrow\Big \},
\end{equation}
and when $Q$ is permutation invariant,
\begin{equation}\label{three infs} 
\inf\Big \{||c-x||: x\in \lambda^{-1}(Q)\Big \}=
\inf\Big \{||\lambda(c)-q||: q\in Q^\downarrow\Big \}=\inf\Big \{||\lambda(c)-q||: q\in Q\Big \}.
\end{equation}

As an illustration, consider the algebra $\Sn$ ($\Hn$). Here, two matrices $X$ and $Y$ operator commute if and only if $XY=YX$ or, 
equivalently, there exist an orthogonal (respectively, unitary) matrix $U$ and real diagonal matrices $D_1$ and $D_2$ such that $X=UD_1U^*$ and $Y=UD_2U^*$.
They strongly operator commute if and only if the above spectral representations hold when the diagonal vectors of 
$D_1$ and $D_2$ are, respectively, $\lambda(X)$ and $\lambda(Y)$. 
To see an explicit example in this setting, 
consider the problem mentioned in the Introduction: Find
$$\sup\big \{ \langle C,X\rangle: X\succeq 0,\, 1\leq \lambda_{max}(X)\leq 2\Big \},$$
where $C,X\in \Sn$, $X\succeq 0$ means that $X$ is positive semidefinite, and 
$\lambda_{max}(X)$ denotes the maximum eigenvalue of $X$. 
Let $Q:=\{q\in \Rn: q\geq 0,\,\,1\leq \max(q)\leq 2\}$, where $\max(q)$ denotes the maximum of the components of $q$. Clearly, $Q$ is permutation invariant and 
$\lambda^{-1}(Q)=\{X\succeq 0,\, 1\leq \lambda_{max}(X)\leq 2\}$. 
Now, (\ref{three sups}) applied to $\Sn$ and $Q$ gives 
$$\sup\Big \{ \langle C,X\rangle: X\in \lambda^{-1}(Q)\Big\}=\sup\Big \{
\langle \lambda(C),q\rangle:q\in Q^\downarrow\Big \}=\sup\Big \{
\langle \lambda(C),q\rangle:q\in Q\Big \}.$$
Here, the  set $Q^\downarrow=\{q=(q_1,q_2,\ldots, q_n)\in \Rn:\,q_1\geq q_2\geq \cdots\geq q_n\geq 0\,\,\mbox{and}\,\,1\leq q_1\leq 2\}$ is polyhedral and compact. So we have the attainment of supremum in both the problems and, moreover, knowing the components of $\lambda(C)$, we can compute this maximum.

\gap

To illustrate (\ref{three infs}), consider a Euclidean Jordan algebra $\V$ and 
let $E:=\lambda^{-1}(Q)$, where $Q$ is  a closed permutation invariant set in $\Rn$. Then, $E$ is a closed spectral set in $\V$, $\lambda(E)=Q\cap Q^\downarrow =Q^\downarrow$ is a closed set in $\Rn$, and 
for any $c\in \V$,
$$
\min\Big \{||c-x||: x\in \lambda^{-1}(Q) \Big \}=
\min\Big \{||\lambda(c)-q||: q\in Q^\downarrow\Big \}=\min\Big \{||\lambda(c)-q||: q\in Q\Big \}.$$
Such a result appears in Proposition 1.11 in \cite{sossa} (or, Proposition 10 in \cite{ramirez et al}). In a similar vein, we specialize (\ref{hausdorff equality}) to  permutation invariant sets $Q$ and $P$ in $\Rn$:
$$d_H\Big (\lambda^{-1}(Q),\lambda^{-1}(P)\Big )=d_H\Big (Q^\downarrow, P^\downarrow\Big )=d_H\Big (Q,P\Big ).$$ 
(It is interesting to observe that this, together with Proposition 1.1 in \cite{gowda-spectral connectedness}, shows that the multivalued map $\lambda^{-1}$ behaves like a linear isometry on the collection of all convex permutation invariant sets.)
\\

Our next result deals with the linear image of a spectral set. 
Suppose $Q$ is a permutation invariant set in $\Rn$. As observed in \cite{baes,jeong-gowda-spectral set,jeong-gowda-spectral cone,gowda-spectral connectedness}, many properties of $Q$ carry over to the spectral set
$\lambda^{-1}(Q)$. In particular,
{\it 
\begin{itemize}
\item [$(i)$] If $Q$ is compact/convex, then $\lambda^{-1}(Q)$ is compact/convex, see \cite{baes},
\item [$(ii)$] If $Q$ is connected, then  $\lambda^{-1}(Q)$ is connected, see \cite{gowda-spectral connectedness},
\item [$(iii)$] If $\V$ is simple and $Q\cap Q^\downarrow$ is connected, then $\lambda^{-1}(Q)$ is connected, see \cite{gowda-spectral connectedness}.
\end{itemize}
}  
\begin{theorem}
{\it Let $\V$ be a Euclidean Jordan algebra of rank $n$ with trace inner 
product. 
Suppose $Q$ is a compact permutation invariant subset of $\Rn$ and one of the following conditions holds:
$(i)$ $Q$ is connected or $(ii)$ $\V$ is simple and $Q\cap Q^\downarrow$ is connected. 
Then, for any $c\in \V$, 
$$\Big \{\langle c,x\rangle:x\in \lambda^{-1}(Q)\Big \}=[\delta,\Delta], $$ where 
$\delta=\langle\, \wl(c),\lambda(a)\rangle$ for some $a\in \lambda^{-1}(Q)$ that strongly operator commutes with $-c$ and
$\Delta=\langle \lambda(c),\lambda(b)\rangle$ for some $b\in \lambda^{-1}(Q)$ that strongly operator commutes with $c$.
}
\end{theorem}

\begin{proof}
 By results stated above, $\lambda^{-1}(Q)$ is compact and connected in $\V$. By the continuity of the function $x\rightarrow \langle c,x\rangle$, 
$\Big \{\langle c,x\rangle:x\in \lambda^{-1}(Q)\Big \}=[\delta,\Delta], $ where
$\delta$ and $\Delta$ are, respectively, the minimum and the maximum of
$\langle c,x\rangle$ over $\lambda^{-1}(Q)$. By our previous results, 
they must be of the form 
 $\delta=\langle\, \wl(c),\lambda(a)\rangle$ for some $a\in \lambda^{-1}(Q)$ that strongly operator commutes with $-c$ and
$\Delta=\langle \lambda(c),\lambda(b)\rangle$ for some $b\in \lambda^{-1}(Q)$ that strongly operator commutes with $c$.
\end{proof}

The following is a simple consequence of the above result. It extends a similar result of Fan, see Corollary 1.6 and Theorem 1.5 in
\cite{tam}.

\begin{corollary}
{\it Consider a simple algebra $\V$ with trace inner product. Then, for any 
$u,c\in \V$, 
\begin{equation}\label{fan type result}
\Big \{ \langle c,x\rangle: x\in [u]\Big \}=[\delta,\Delta],
\end{equation}
where $\delta=\langle\, \wl(c),\lambda(u)\rangle$ and $\Delta=\langle \lambda(c),\lambda(u)\rangle.$
}
\end{corollary}

\begin{proof}
 We let $Q:=\{\lambda(u)\}$. Then, condition $(ii)$ in the above theorem applies.
\end{proof}

We conclude this section with some remarks about Theorems  \ref{vi result} and \ref{differentiable opt problem in ftvn}. 
Specialized to a Euclidean Jordan algebra $\V$, Theorem \ref{vi result}
 says that if $E$ is a spectral set in $\V$ and $a$ solves the variational inequality problem VI$(G,E)$, then $a$ and $-G(a)$ strongly operator commute. This, in particular, implies that $a$ and $G(a)$ operator commute, thus yielding  
a result of Ram\'{i}rez, Seeger, and Sossa (\cite{ramirez et al}, Proposition 8).

Theorem \ref{differentiable opt problem in ftvn} stated in the setting of $\V$ gives the strong operator commutativity of 
$a$ and $-h^\prime(a)$. This, in particular, gives the operator commutativity of $a$ and $h^\prime(a)$. 
We note that if  one is concerned with just the operator commutativity of $a$ and $h^{\prime}(a)$, then,
the following result -- where no convexity assumptions are made -- can be used.
This result extends Theorem 1.2 in \cite{gowda-jeong} with a similar/modified proof. Here, weak spectrality refers to invariance under automorphisms of the $\V$, see \cite{gowda-jeong}. 

\begin{theorem}\label{ramirez et al}
{\it Let $\V$ be a Euclidean Jordan algebra  and suppose that $E$ is a (weakly) spectral set in $\V$ and $\Phi$ is
 a (weakly) spectral function on $\V$. Let $L$ (defined on a product of appropriate intervals in $\R$) be strictly increasing in the first variable and $h:\V\rightarrow \R$ be Fr\'{e}chet differentiable. If $a$ is a local optimizer of the problem $\min_{E} L(h,\Phi)$ or $\max_{E}L(h,\Phi)$, then $a$ and $h^{\prime}(a)$ operator commute in $\V$.
}
\end{theorem}

The following example shows that operator commutativity cannot be replaced by strong operator commutativity.

\gap

\begin{example} 
In the Euclidean Jordan algebra $\R^2$, let $E=\{(1,0),\,(0,1)\}$. For the function
$h(x,y):=\frac{1}{2}x^2-x+x(y^2+y)$, we have $h(1,0)=-\frac{1}{2}$ and $h(0,1)=0$. 
Also, $h^{\prime}(x,y)=(x-1+y^2+y,2xy+x)$. So, $h^{\prime}(1,0)=(0,1)$ and $h^{\prime}(0,1)=(1,0)$. We note that the elements $(1,0)$ and $(0,1)$ operator 
commute in $\R^2$, but not strongly. Thus, if  $a$ denotes either a minimizer or a maximizer of $h$ on $E$, then $a$ and $h^\prime (a)$ do not strongly operator commute.
\end{example}
 
\section{FTvN systems induced by hyperbolic polynomials}
Hyperbolic polynomials were introduced by G\aa rding \cite{garding} in connection with partial differential equations.
They have become  important in optimization due to their connection to interior point methods \cite{guler,renegar} and convex analysis \cite{bauschke et al}.  

Let $\V$ be a finite dimensional real vector space. With respect to some coordinate system on $\V$, let $p:\V\rightarrow R$ be a nonconstant polynomial, homogeneous of degree $n$, where $n$ is a natural number. We say that {\it $p$ is hyperbolic with respect to some $e\in \V$ if $p(e)\neq 0$ and for all $x\in \V$, the univariate polynomial
$t\rightarrow p(te-x)$ has only real roots.} 
We fix such a $p$ and consider, to each $x\in \V$, the vector $\lambda(x)$ in $\Rn$ whose entries are the  of roots  of $p(te-x)$ written in the decreasing order. Then the map $\lambda:\V\rightarrow \Rn$ has many interesting properties \cite{bauschke et al}. Assuming that $p$ is {\it complete}, that is, $\lambda(x)=0\Rightarrow x=0,$
one can define a norm on $\V$ by $||x||:=||\lambda(x)||$ (the latter norm is the Euclidean norm on $\Rn$) and an inner product on $\V$ by $\langle x,y\rangle:=\frac{1}{4}[||x+y||^2-||x-y||^2]$, see \cite{bauschke et al}. Then, for all $x,y\in \V$,
$$\langle x,y\rangle \leq \langle \lambda(x),\lambda(y)\rangle,$$
see \cite{bauschke et al}, Prop. 4.4.
In connection with the triple $(\V,\Rn,\lambda)$,  we have the following result.

\begin{proposition} (Proposition 5.3, \cite{bauschke et al})
{\it Assume that $p$ is complete and consider the associated norm/inner product on $\V$. Then, the following are equivalent:
\begin{itemize}
\item [$(i)$] $p$ is `isometric', that is, for all $y,z\in \V$, there is an $x\in \V$ such that $\lambda(x)=\lambda(z)$ and 
$\lambda(x+y)=\lambda(x)+\lambda(y).$
\item [$(ii)$] $\max_{\{x:\lambda(x)=\lambda(u)\}}\langle c,x\rangle=\langle \lambda(c),\lambda(u)$ for all $c,u\in \V$.
\end{itemize}
}\end{proposition}

Given that $\lambda$ is norm preserving and the inequality 
$\langle x,y\rangle \leq \langle \lambda(x),\lambda(y)\rangle$ holds for all $x,y\in \V$, Item $(ii)$ above proves (see Item $(a)$ in Proposition \ref{equivalence of linear and distance functions}) the following.

\begin{theorem}
{\it 
Assume that $p$ is complete and consider the induced triple $(\V,\Rn,\lambda)$. Then,
$p$ is isometric if and only if $(\V,\Rn,\lambda)$ is a FTvN system.
}
\end{theorem}

We observe that the condition $\lambda(x+y)=\lambda(x)+\lambda(y)$ that appears in Item $(i)$ of the previous proposition is simply a restatement of the commutativity condition $\langle x,y\rangle=\langle \lambda(x),\lambda(y)\rangle$.
There are numerous examples of complete isometric hyperbolic polynomials, see \cite{bauschke et al}. We provide one more example.
\\
\begin{example} 
Consider a Euclidean Jordan algebra $\V$ of rank $n$ with the trace inner product. 
Then, for any $x\in \V$, the determinant of $x$ (written $\det x$) is the product of eigenvalues of $x$ \cite{faraut-koranyi}. 
Let $p(x):=\det x$. With $e$ denoting the unit element in $\V$, the roots of the univariate polynomial $t\rightarrow p(te-x)$ are precisely the eigenvalues of $x$. We have already seen (in the previous section) that 
$(\V,\Rn,\lambda)$ is a FTvN system. Hence the corresponding $p$ is isometric.
As $\lambda(x)=0\Rightarrow x=0$, this $p$ is also complete.
\end{example}

\noindent {\bf Remarks.} The problem of characterizing the `isometric' property of a complete hyperbolic polynomial seems open. 
However, the commutativity condition $\lambda(x+y)=\lambda(x)+\lambda(y)$ can be described as follows.  Based on the validity of the Lax conjecture \cite{lewis-lax}, Gurvits \cite{gurvits} has shown that in the canonical setting of $\V=\Rn$ and $p(e)=1$, 
for any two elements $x,y$, there exist real symmetric $n\times n$ matrices $A$ and $B$ such that 
$$\lambda(tx+sy)=\lambda(tA+sB),$$
for all $t,s\in \R$, where the right-hand side denotes the eigenvalue vector of a symmetric matrix. This result, in addition to showing  Lidskii type (majorization) results in the setting of  hyperbolic polynomials (\cite{gurvits}, Corollary 1.3) 
hence in all Euclidean Jordan algebras, also yields a description of commutativity: 
$$\lambda(x+y)=\lambda(x)+\lambda(y)\Longleftrightarrow \lambda(A+B)
=\lambda(A)+\lambda(B)\Longleftrightarrow A\,\mbox{and}\,B\,\mbox{strongly operator commute in}\,\Sn.$$ 
\section{Normal decomposition systems and Eaton triples}
Normal decomposition systems were introduced by Lewis \cite{lewis} to 
 unify various results of convex analysis. 
One key assumption in the definition of normal decomposition system (see below) is an inequality that is similar to 
the one that appears in Theorem \ref{baes}. Another feature is the use of orthogonal transformations 
in place of Jordan frames.

\begin{definition} \cite{lewis}\label{definition: nds}
{\it Let $\V$ be a real  inner product space, $\G$ be a closed subgroup of the orthogonal group of $\V$, and 
$\gamma : \V \rightarrow \V$ be a map satisfying the following conditions:
\begin{itemize}
        \item [$(a)$] $\gamma$ is $\G$-invariant, that is, $\gamma(Ax) = \gamma(x)$ for all $x \in \V$ and $A \in \G$.
        \item [$(b)$] For each $x\in \V$, there exists $A\in \G$ such that $x=A\gamma(x)$. 
        \item [$(c)$] For all $x,y\in \V$, we have $\langle x,y\rangle \leq \langle \gamma(x),\gamma(y)\rangle$.
\end{itemize}
Then, $(\V,\, \G,\, \gamma)$ is called a {\it normal decomposition system}. 
}\end{definition}
 
\gap

Items $(a)$ and $(b)$ in the above definition show that $\gamma^2=\gamma$ and $||\gamma(x)||=||x||$ for all $x$. 
We state a few relevant properties.

\begin{proposition} [\cite{lewis}, Proposition 2.3 and Theorem 2.4] \label{prop: lewis}
{\it Let $(\V,\G,\gamma)$ be a normal decomposition system. Then,
\begin{itemize}
\item [$(i)$] For any two elements $x$ and $y$ in $\V$, we have
        $$ \max_{A \in \G}\, \langle Ax,y\rangle = \langle \gamma(x),\gamma(y)\rangle. $$
        Also, $\langle x,y\rangle = \langle \gamma(x),\gamma(y)\rangle $  holds for two elements $x$ and $y$
if and only if there exists an $A \in \G$ such that $x = A\gamma(x)$ and $y = A\gamma(y)$.
\item [$(ii)$] The range of $\gamma$, denoted by  $F$, is a closed convex cone in $\V$.
\end{itemize}
}
\end{proposition}

\gap 

{\it Eaton triples} were introduced and studied in \cite{eaton-perlman,eaton1,eaton2} 
from the perspective of  majorization techniques in probability. They were also extensively studied in the papers of Tam and Niezgoda, see the references. 

\begin{definition} \label{definition: eaton triple}
{\it Let $\V$ be a finite dimensional real  inner product space, $\G$ be a closed subgroup of the orthogonal group of $\V$, and
$F$ be  a closed convex cone in $\V$ satisfying  the following conditions:
\begin{itemize}
        \item [$(a)$] $Orb(x)\cap F\neq \emptyset$ for all $x\in \V$, where $Orb(x):=\{Ax:A\in \G\}$.
        \item [$(b)$] $\langle x,Ay\rangle \leq \langle x,y\rangle$ for all $x,y\in F$ and $A\in \G$.
\end{itemize}
Then, $(\V,\, \G,\, F)$ is called an {\it Eaton triple}.
}\end{definition}

It has been shown (see \cite{niezgoda-group majorization}, page 14) that  in an Eaton triple $(V,\, \G,\, F)$, $Orb(x)\cap F$ consists of exactly one element for each 
$x\in \V$. Defining $\gamma:\V\rightarrow \V$ such that $Orb(x)\cap F=\{\gamma(x)\}$, it has been observed 
that $(\V,\G, \gamma)$ is a normal decomposition system.
Also, given a finite dimensional normal decomposition system $(\V,\G, \gamma)$, with $F:=\gamma(\V)$, 
$(V,\, \G,\, F)$ becomes an Eaton triple. Thus, {\it finite dimensional normal decomposition systems are equivalent to Eaton triples} \cite{lewis,lewis3,niezgoda-commutation}.
 
While both appear in various matrix and Lie algebraic
settings \cite{tam}, in this paper, we state our results (only) for normal decomposition systems. 

\begin{theorem}
{\it Let $(\V,\G,\gamma)$ be a normal decomposition system and $\W:=span(\gamma(\V))$. Then,  $(\V,\W,\gamma)$ is a
 FTvN system.}
\end{theorem}

\begin{proof}
We verify conditions $(A1)-(A3)$ in Definition \ref{ftvn}. The norm preserving property of $\gamma$ comes from Item $(b)$ in 
Definition \ref{definition: nds}. Thus, condition $(A1)$ holds. Because of Item $(c)$ in Definition \ref{definition: nds}, we have $(A2)$. Suppose $c\in \V$ and $q\in \gamma(\V)$. Let $q=\gamma(u)$. From Item $(i)$ in Proposition \ref{prop: lewis}, we have, for some $A\in \G$,
$\langle c,Au\rangle=\langle \gamma(c),\gamma(u)\rangle.$ Letting $x=Au$, we observe that $\gamma(x)=\gamma(Au)=\gamma(u)=q$ and so $\langle c,x\rangle=\langle \gamma(c),q\rangle$. This verifies condition $(A3)$.
\end{proof}

\gap

Now, let $(\V,\G,\gamma)$ be a normal decomposition system. Using the notation of Section 2,  
for any $u\in \V$,
$$[u]=\{x:\gamma(x)=\gamma(u)\}=\{Au: A\in {\cal G}\}=:\mbox{Orb}(u).$$
In view of the remarks made after Definition \ref{spectral set}, we see that  in the 
FTvN system $(\V,\W,\gamma)$, a set $E$ is spectral if and only if it is $\G$-invariant, that is, for all $A\in \G$, $A(E)\subseteq E$. Similarly, a function $\Phi:\V\rightarrow \R$ is spectral if and only  if it is $\G$-invariant, that is, for all $A\in \G$ and $x\in \V$, $\Phi(Ax)=\Phi(x)$. 
(In some literature, $\G$-invariant functions are called orbital functions \cite{tam-hill}.)\\

We recall the concept of  commutativity in a normal decomposition system.

\begin{definition} \cite{gowda-jeong} In a normal decomposition system $(\V,\G, \gamma)$, we say that $x$ and $y$ {\it commute } if there exists an $A \in \G$ such that $x = A\gamma(x)$ and $y = A\gamma(y)$.
\end{definition}

Analogous to Proposition \ref{strong operator commutativity in eja}, we have the following.

\begin{proposition}\label{commutativity in nds}
{\it For elements $x$ and $y$ in  a normal decomposition system $(\V, \G,\gamma)$, the following are equivalent:
\begin{itemize}
\item [$(i)$] $x$ and $y$ commute in the normal decomposition system $(\V, \G,\gamma)$.
\item [$(ii)$] $\gamma(x+y)=\gamma(x)+\gamma(y)$.
\item [$(iii)$] $\langle x,y\rangle=\langle \gamma(x),\gamma(y)\rangle$, that is, $x$ and $y$ commute in the FTvN system $(\V, \W,\gamma)$.
\end{itemize}
}
\end{proposition}
  
\begin{proof}
When $(i)$ holds, we can write
$x=A\gamma(x)$ and $y=A\gamma(y)$ for some $A\in \G$. 
Letting $z=\gamma(x)+\gamma(y)$, we have 
$x+y=A(\gamma(x)+\gamma(y))=Az$.  
Since  the range of $\gamma$ is a closed convex cone 
(see item $(ii)$ in Proposition \ref{prop: lewis}), we can write $z=\gamma(u)$ for some $u\in \V$. 
Then, $x+y=A\gamma(u)$.  As $\gamma^2=\gamma$, we have  
$\gamma(x+y)=\gamma(A\gamma(u))=\gamma(\gamma(u))=\gamma(u)=z=\gamma(x)+\gamma(y)$. This proves the implication 
$(i)\Rightarrow (ii)$. \\
The equivalence of $(ii)$ and $(iii)$ comes from Proposition 
\ref{commutativity proposition}.\\
Finally, the equivalence of $(iii)$ and $(i)$ follows from Proposition \ref{prop: lewis}.
\end{proof}

\gap

We now specialize some results of Section 3. Consider a normal decomposition system $(\V,\G,\gamma)$,  let $F:=\gamma(\V)$ 
(which is a closed convex cone) and $W:=span(F)$ ($=F-F$). Consider any $Q\subseteq \V$. As $\gamma^2=\gamma$, 
$$Q\cap F=Q\cap \gamma(Q),\,\,\gamma^{-1}(Q)=\gamma^{-1}(Q\cap F)\,\,\mbox{and}\,\,\gamma(\gamma^{-1}(Q))=Q\cap\gamma(Q).$$
From Corollary \ref{linear sup-inf corollary} with  $E=\gamma^{-1}(Q)$, $\Phi=0$, and $c\in \V$, we get
$$
\sup\Big \{\langle c,x\rangle: x\in \gamma^{-1}(Q)\Big \}=\sup\Big \{\langle \gamma(c),q\rangle: q\in Q\cap\gamma(Q)\Big \}.
$$
Moreover, attainment of supremum in one problem implies that in the other and the maximum value is given by $\langle \gamma(c),\gamma(a)\rangle$ for some $a\in E$ that commutes with $c$ in $(\V,\G,\gamma)$.
\\
Similarly, from Corollary \ref{distance sup-inf corollary}, we get 
$$\inf\Big \{||c-x||: x\in \gamma^{-1}(Q)\Big \}=\inf\Big \{||\gamma(c)-q||: q\in Q\cap\gamma(Q)\Big \}.$$
Moreover, attainment of infimum in one problem implies the attainment in the other, and the minimum value is given by
$\langle \wg(c),\gamma(a)\rangle$ for some $a\in E$ that commutes with $-c$ in $(\V,\G,\gamma)$.

Specialized, we can now recover the results of von Neumann, Chu and Driessel, Tam mentioned in the Introduction. 

\gap

The papers by Lewis, Eaton, Eaton and Perlman, Lim et al, Niezgoda,  and Tam (see the References) contain numerous examples of normal decomposition systems (Eaton triples) 
related to matrices,  Lie and Euclidean Jordan algebras. We specifically note that
the space $M_n$ of $n\times n$ complex matrices \cite{lewis} and any simple Euclidean Jordan algebra \cite{lim et al} are examples of normal decomposition systems.
We now provide an example of a FTvN system that is neither a  normal decomposition system nor a Euclidean Jordan algebra.
Let $\V$ be any real inner product space whose dimension is more than one. On $\V$, let  $S:\V\rightarrow \V$ be linear and
inner product preserving (so it is an  isometry, but need not be surjective).
We assume that $S$ is  different from the identity transformation.
Then, with $\gamma(x):=Sx$, the triple $(\V,\V,\gamma)$ is a FTvN system. We claim that this is not a normal decomposition system. Suppose, if possible,  $(\V,\V,\gamma)$ is a normal decomposition system so that there is a closed subgroup ${\cal G}$  of the group of orthogonal transformations on $\V$ satisfying Definition \ref{definition: nds}. Then $\gamma(Ax)=\gamma(x)$ for all $x\in \V$ and $A\in {\cal G}$; so, $SAx=Sx$ for all $x$. As $S$ is injective, $Ax=x$ for all $x\in \V$ implying that
 ${\cal G}$ consists only of the identity transformation. But then, by  condition $(b)$ in Definition \ref{definition: nds}, $x=\gamma(x)=Sx$. As $S$ is not the identity transformation, we have a contradiction. Thus, $(\V,\V,\gamma)$ is not a normal decomposition system.
Specializing, let $\V=\R^2$ and $S:\R^2\rightarrow \R^2$ be rotation through $90^\circ$ about the origin.
Then $(\R^2,\R^2,S)$ is a FTvN system. If this were of the form $(\V,\Rn,\lambda)$ for some Euclidean Jordan algebra, then, $\V=\R^2$, $n=2$, and $\lambda=S$. But in this setting, $\lambda^2=\lambda$. This implies that  $S^2= S$, which is false. Hence, $(\R^2,\R^2,S)$ is a FTvN system which is neither a normal decomposition system nor a Euclidean Jordan algebra.



\begin{thebibliography}{}

\bibitem{baes} M. Baes, \emph{Convexity and differentiability properties of spectral functions in Euclidean Jordan algebras,}
Linear Algebra Appl., 422 (2007) 664-700.

\bibitem{bauschke et al} 
H.H. Bauschke, O. G\"{u}ler, A.S. Lewis, and H.S. Sendov, \emph{Hyperbolic polynomials and convex analysis}, Canadian J. Math., 53 (2001) 470-488.

\bibitem{chu-driessel}
M.T. Chu and K.R. Drissel, \emph{The projected gradient method for least squares matrix approximations and spectral constraints}, SIAM J. Numer. Anal., 27 (1990) 1050-1060.

\bibitem{facchinei-pang}  
F. Facchinei and J.-S. Pang, \emph{Finite Dimensional Variational Inequalities and Complementarity Problems, Volumes I \& II}, Springer, New York, 2003.

\bibitem{eaton1} M.L. Eaton, \emph{On group induced orderings, monotone functions, and convolution theorems,
in: Inequalities in Statistics and Probability}, Y. L. Tong (ed.), IMS Lectures Notes, Monograph Series 5, 
IMS, Hayward (1984) 13-25.

\bibitem{eaton2} M. L. Eaton, \emph{Group induced orderings with some applications in statistics}, CWI Newsletter,
16 (1987) 3-31.


\bibitem{eaton-perlman} M.L. Eaton and M.D. Perlman, \emph{Reflection groups, generalized Schur functions, and the
geometry of majorization}, Ann.  Probab., 5 (1977) 829-860.

\bibitem{ekeland-temam} I. Ekeland and R. Temam, \emph{Convex Analysis and Variational Problems}, North-Holland, New York, 1976.

\bibitem{fan} K. Fan, \emph{On a theorem of Weyl concerning eigenvalues of linear transformations}, Proc. Natl. Acad. Sci. USA, 35 (1949) 652-655.

\bibitem{faraut-koranyi}
J. Faraut and A. Kor\'{a}nyi, \emph{Analysis on Symmetric Cones}, Clarendon Press, Oxford, 1994.

\bibitem{garding} L. G\aa rding, \emph{An inequality for hyperbolic polynomials,} J. Math. Mech. 8 (1959) 957-965.

\bibitem{gowda-jeong} M.S. Gowda and J. Jeong, \emph{Commutation principles in Euclidean Jordan algebras and normal decomposition systems}, SIAM J. Optim., 27 (2017) 1390-1402.


\bibitem{gowda-spectral connectedness}
M.S. Gowda and J. Jeong, \emph{On the connectedness of spectral sets and irreducibility of spectral cones in Euclidean Jordan algebras}, Linear Algebra Appl., 559 (2018) 181-193.


\bibitem{gowda-tao} M.S. Gowda and J. Tao, \emph{
The Cauchy interlacing theorem in simple Euclidean Jordan algebras and some consequences},
 Linear and Multilinear Algebra, 59 (2011) 65-86.

\bibitem{guler} O. G\"{u}ler, \emph{Hyperbolic polynomials and interior point methods for convex programming,} Math. Oper. Res., 22 (1997) 350-377.

\bibitem{gurvits} L. Gurvits, \emph{Combinatorics hidden in hyperbolic polynomials and related topics}, arXiv:math/0402088v1 [math.CO], 2004.

\bibitem{holmes-tam} R.R. Holmes and T.-Y. Tam, \emph{Distance to the convex hull of an orbit under the action of a compact Lie group}, J. Austral. Math. Soc. (Series A), 6 (1999) 331-357. 

\bibitem{jeong-gowda-spectral cone}
J. Jeong and M.S. Gowda, \emph{Spectral cones in Euclidean Jordan algebras},
Linear Algebra Appl., 509 (2016) 286-305.

 
\bibitem{jeong-gowda-spectral set}
J. Jeong and M.S. Gowda, \emph{Spectral sets and functions in Euclidean Jordan algebras}, Linear Algebra Appl., 518 (2017) 31-56.


\bibitem{lewis}
A. S. Lewis, \emph{Group invariance and convex matrix analysis}, SIAM J.
 Matrix Anal., 17 (1996) 927-949.

\bibitem{lewis2}
A. S. Lewis, \emph{Convex analysis on the Hermitian matrices}, SIAM J. Optim.,
6 (1996) 164-177.

\bibitem{lewis3}
A.S. Lewis, \emph{Convex analysis on Cartan subspaces}, Nonlinear Anal., 42 (2000) 813-820.

\bibitem{lewis-lax}
A. S. Lewis, P.A. Parrilo, and M. V. Ramana, \emph{The Lax conjecture is true},
Proc. Amer. Math. Soc., 133 (2005) 2495-2499.


\bibitem{li-tsing}
C.K. Li and N.K. Tsing, \emph{Distance to the convex hull of the unitary orbit with respect to the unitary similarity invariant norms}, Linear and Multilinear Algebra, 25 (1989) 93-103.


\bibitem{lim et al}
Y. Lim, J. Kim, and L. Faybusovich,
\emph{Simultaneous diagonalization on simple Euclidean Jordan algebras and its applications}, Forum Math., 15 (2003) 639-644.

\bibitem{marshall-olkin} A.W. Marshall and  I. Olkin, \emph{Inequalities: Theory of Majorization and its Applications}, Academic Press, New York (1979).

\bibitem{megginson} R.E. Megginson, \emph{An Introduction to Banach Space Theory}, Springer, New York (1998).

\bibitem{moldovan} M. Moldovan, \emph{A Gersgorin type theorem, spectral inequalities, and simultaneous stability in
Euclidean Jordan algebras}, PhD Thesis, University of Maryland, Baltimore County, 2009.

\bibitem{niezgoda-group majorization} M. Niezgoda, \emph{Group majorization and Schur type inequalities}, Linear Algebra 
Appl., 268 (1998) 9-30.

\bibitem{niezgoda-commutation} M. Niezgoda, \emph{Extended commutation principles for normal decomposition systems}, Linear Algebra 
Appl., 539 (2018) 251-273.

\bibitem{ramirez et al} H. Ram\'{i}rez, A. Seeger, and D. Sossa, \emph{Commutation  principle  for  variational  problems  on
Euclidean Jordan algebras}, SIAM J. Optim., 23 (2013) 687-694.

\bibitem{renegar} J. Renegar, \emph{
Hyperbolic programs, and their derivative relaxations},
Found. Comput. Math., 6 (2006) 59-79.

\bibitem{rockafellar} R.T. Rockafellar, \emph{Convex Analysis}, Princeton University Press, Princeton, 1970. 


\bibitem{sossa} D.  Sossa, \emph{Euclidean  Jordan  algebras  and  variational  problems  under  conic  constraints}, Ph.D. Thesis, University of Chile, 2014.

\bibitem{tam} T.-Y. Tam, \emph{An extension of a result of Lewis}, Elec. J. Linear Algebra, 5 (1999) 1-10.

\bibitem{tam-hill}  T.-Y. Tam and W.C. Hill, \emph{Derivatives of  orbital function and an extension of Berezin-Gel'fand's theorem}, Spec. Matrices, 4 (2016) 333-349.

\bibitem{theobald} C. M. Theobald, \emph{An inequality for the trace of the product of two symmetric matrices,}
Math. Proc. Camb. Philos. Soc., 77 (1975) 265-267.
\end{thebibliography}
\end{document}